\documentclass[a4paper, 11pt]{article}
\usepackage[margin=1in]{geometry}
\usepackage{graphicx} 
\usepackage{amsthm}
\usepackage[toc,page]{appendix}
\usepackage{soul}
\usepackage{bbm}
\usepackage{tikz}
\usetikzlibrary{automata, positioning, arrows}
\usepackage{amsmath,amssymb,bm}
\usepackage{hyperref}
\usepackage[capitalise]{cleveref}
\usepackage[skip=6pt, indent=20pt]{parskip}
\usepackage{lineno}

\renewcommand{\P}{\mathbb{P}}
\newcommand{\E}{\mathbb{E}}

\newcommand{\allone}{\mathbf{1}}

\newcommand{\ind}{\mathbf{1}}

\newcommand{\Ell}{\mathcal{L}}

\newcommand{\diag}{\textnormal{diag}}
\newcommand{\inv}{\tau_\textnormal{inv}}

\newcommand{\compi}{\textnormal{i}}

\newcommand{\cov}{\textnormal{cov}}
\newcommand{\tr}{\textnormal{tr}}
\newcommand{\orth}{{\mathbb{R}^{2n}}}
\newcommand{\orthp}{{\mathbb{R}_+^n}}

\newtheorem{thm}{Theorem}[section]
\newtheorem{lem}[thm]{Lemma}
\newtheorem{cor}[thm]{Corollary}
\newtheorem{rem}{Remark}[thm]

\title{The Density Formula Approach for Non-reversible Isomorphism Theorems, with Applications}
\author{Qinghua (Devon) Ding, \quad Venkat Anantharam\\
  Department of Electrical Engineering and Computer Sciences\\
                    University of California at Berkeley\\
                    Berkeley, CA, United States\\
Email: \{devon\_ding, ananth\}@berkeley.edu}

\begin{document}

\maketitle

\begin{abstract}
The classical isomorphism theorems for reversible Markov chains have played an important role in studying the properties of local time processes of strongly symmetric Markov processes~\cite{mr06}, bounding the cover time of a graph by a random walk~\cite{dlp11}, and in topics related to physics such as random walk loop soups and Brownian loop soups~\cite{lt07}. Non-reversible versions of these theorems have been discovered by Le Jan, Eisenbaum, and Kaspi~\cite{lejan08, ek09, eisenbaum13}. Here, we give a density-formula-based proof for all these non-reversible isomorphism theorems, extending the results in \cite{bhs21}. Moreover, we use this method to generalize the comparison inequalities derived in \cite{eisenbaum13} for permanental processes and derive an upper bound for the cover time of non-reversible Markov chains.
\end{abstract}

\section{Introduction}

Since the discovery of the Dynkin-BFS isomorphism theorem~\cite{dynkin84, bfs82}, isomorphism theorems for reversible Markov chains have been an important tool to study stochastic processes related to Markov chains. Examples of applications of these isomorphism theorems include the study of large deviations, continuity, and boundedness of local time processes of Markov processes~\cite{ek09, mr06}, tight bounds for and concentration of the cover time of
random walks on graphs~\cite{dlp11, zhai18}, percolation and level sets of the loop soup ensemble~\cite{lupu16}, etc. The book \cite{mr06} contains a rather diverse list of such applications.

Non-reversible versions of Dynkin's isomorphism theorem and the generalized second Ray-Knight isomorphism theorem have been discovered by Le Jan, Eisenbaum, and Kaspi~\cite{lejan08,ek09, eisenbaum13}. These isomorphism theorems were used to study properties of the local time process for non-reversible Markov chains, such as continuity~\cite{mr13} and boundedness~\cite{mr17} and also to study the large deviations of local 
Ward identities as in~\cite{bhs21},
together with these density formulas, to give succinct proofs for the non-reversible isomorphism theorems.

We then use the density formula to generalize the comparison inequalities derived in~\cite{eisenbaum13} for permanental processes. Besides, we also derive an upper bound for the cover time of non-reversible Markov chains, using the non-reversible
generalized second Ray-Knight isomorphism theorem together with the density formulas for conditioned 1-permanental processes that we prove.

\section{Density of 1-permanental vectors}

We start with some basic notational conventions. For a positive integer $n$, let $[n]$ denote $\{1, \ldots, n\}$. Let $M_n(\mathbb{R})$ denote the set of $n \times n$ matrices with real entries and $M_n(\mathbb{R}_+)$ the set of $n \times n$ matrices with nonnegative real entries. We write $:=$ for equality by definition.

We say that $G \in M_n(\mathbb{R})$ is diagonally equivalent to $\tilde{G} \in M_n(\mathbb{R})$ if there exists a diagonal matrix $D \in M_n(\mathbb{R})$ with non-zero diagonal entries such that $\tilde{G}=D^{-1}GD$.
Given $A\in M_n(\mathbb{R})$, its spectral radius is defined as $\rho(A):=\lim_{k\rightarrow\infty}\|A^k\|^{\frac{1}{k}}$, where $\|\cdot \|$ is any matrix norm (the choice of the norm does not matter). We say that $A\in M_n(\mathbb{R})$ is diagonally dominant if $|A_{ii}|\geq\sum_{j\neq i}|A_{ij}|$ for all $i\in[n]$. It is called strictly diagonally dominant if these inequalities are strict for all $i\in[n]$. $A\in M_n(\mathbb{R})$ is called an $M$-matrix if it can be written as $sI- B$ for matrix $B$ with nonnegative entries, and where $s \ge \rho(B)$, see \cite[pg. 176]{plemmons77}. It is known that
$A$ is a non-singular M-matrix iff:
\begin{enumerate}
    \item $A_{ij}\leq 0$ for $i\neq j$;
    \item $A$ is non-singular and $A^{-1}$ is entry-wise nonnegative.
\end{enumerate}
See \cite[Criterion F15, pg. 180]{plemmons77}.

For any $\alpha>0$, a random vector $\ell\in\mathbb{R}_+^n$ is called an $\alpha$-permanental vector with kernel $G \in M_n(\mathbb{R})$ if, for
every $\lambda\in\mathbb{R}_{+}^n$, we have the Laplace transform
\[\E[e^{-\langle \lambda, \ell\rangle}]=|I+\Lambda G|^{-\alpha}.\]
Here, on the right-hand side, $\Lambda := \mbox{diag}(\lambda_1, \ldots, \lambda_n)$, and $|I+\Lambda G|$ denotes the determinant of $I+\Lambda G$. We remark that a 1-permanental vector can correspond to different kernels (e.g., arising from diagonal equivalence). Vere-Jones~\cite{verejones97} has established necessary and sufficient conditions on the pair $(G, \alpha)$ for $G$ to be the kernel of an $\alpha$-permanental vector, see also \cite[Sec. 2]{eisenbaum17}.

1-permanental vectors are closely related to non-reversible isomorphism theorems \cite{lejan08, ek09}, and
will be the main object of our study.
Our first
contribution is a density formula for a broad range of 1-permanental vectors.
The proof is based on characteristic functions, following an approach similar to that of Brydges-Hofstad-König~\cite{bhk07}.

In order to state the result, we need to introduce some conventions regarding complex numbers and vectors. We write $\compi$ for $\sqrt{-1}$, to distinguish it from a running index $i$. For a complex number $\phi\in\mathbb{C}$, we let $\Re(\phi)$
denote the real part, and $\Im(\phi)$
the imaginary part. We write $\phi =u +\compi v$, so that $u = \Re(\phi)$ and $v = \Im(\phi)$. Then $\bar{\phi} = u - \compi v$ denotes the complex conjugate of $\phi$. For $\phi, \psi \in \mathbb{C}^n$, let $\langle \phi, \psi \rangle := \sum_{i=1}^n \phi_i \psi_i$ denote the (indefinite) real inner product on $\mathbb{C}^n$. Note that we do not conjugate either argument in computing this inner product. It will also be convenient to think of $\phi \in \mathbb{C}$ as $\sqrt{l} e^{\compi \theta}$, so that $\sqrt{l} =|\phi|$ denotes the absolute value of $\phi$ and $\theta$ denotes its phase. Note that, in one complex dimension, we have $du dv = 2^{-1} dl d\theta$. Also observe that we have, in one complex dimension, the formula
\begin{equation}        \label{eq:diffform}
d \bar{\phi} \wedge d \phi = (du - \compi dv) \wedge (du + \compi dv) = 2 \compi du \wedge dv = \compi dl \wedge d\theta,
\end{equation}
as an equality of differential forms.

We can now state our density formula for the density of a class of 1-permanental vectors. Here, for any matrix $M\in M_n(\mathbb{R})$, its symmetric part is defined to be $\frac{1}{2}(M + M^T)$.

\vspace{.3em} \begin{lem}[Density of 1-permanental vectors]   \label{lem:density}
Let $G\in M_n(\mathbb{R})$ have a positive definite symmetric part.
If there exists a 1-permanental process $\ell$ with kernel $G$, then the density $\rho(l)$ can be represented as follows:
\begin{equation}
\rho(l) = \int_{(0,2\pi)^n}\frac{1}{(2\pi)^n|G|}\, e^{\sum_{i,j} Q_{ij}\sqrt{l_i l_j}e^{\compi (\theta_i-\theta_j)}}d\theta.
\label{eqn:densityformula}
\end{equation}

Here $Q :=-G^{-1}$. Also $l_i$ denotes the $i$-th component of the vector
$l$, and we write $d\theta$ for $\prod_i d \theta_i$.
\qed
\end{lem}

\vspace{.5em} \vspace{.3em} \begin{rem} \label{rem:rem1-lem:density}
It is straightforward to check that if $G\in M_n(\mathbb{R})$ has a positive definite symmetric part, then $G$ is invertible.
\qed
\end{rem}

\vspace{.3em} \begin{rem} \label{rem:notnew-lem:density}
Lemma \ref{lem:density} is not completely new. It appeared in \cite{av20}, but only in the Markovian setting (defined in \cref{rem:rem2-lem:density} below). Using the results derived by Brydges-Hofstad-König~\cite{bhk07},
our contribution is to generalize the density formula of \cite{av20} to any kernel with a positive definite symmetric part.
\qed
\end{rem}

\vspace{.3em} \begin{rem}[Infinitely divisible kernels]
\label{rem-infdiv}
By the definition of infinite divisibility, a 1-permanental vector $\ell$ is infinitely divisible if and only if, for every positive integer $n$, there exist $n$ i.i.d. random vectors $\ell_1$, ..., $\ell_n$ such that $\tilde{\ell}:= \ell_1 + \cdots + \ell_n$ is equal in distribution to $\ell$. We say that a kernel $G$ of a $1$-permanental vector is infinitely divisible if the corresponding 1-permanental vector $\ell$ is infinitely divisible. In \cite{ek09}, several characterizations are given of the family of 
kernel $G$ of an infinitely divisible 1-permanental vector is diagonally equivalent to an inverse M-matrix, i.e., the inverse of an invertible M-matrix $A$.
\footnote{If $G$ is
a kernel of an infinitely divisible $1$-permanental vector then there should be a $\frac{1}{2}$-permanental vector $\ell$ with kernel $G$.
This random vector $\ell$ must be infinitely divisible. On the other hand, if $\ell$ is a $\frac{1}{2}$-permanental vector with kernel $G$ such that $\ell$ is infinitely divisible, then the corresponding 1-permanental vector should also be infinitely divisible, and hence $G$ is an infinitely divisible kernel. According to \cite[Lemma 4.2]{ek09}, if a $\frac{1}{2}$-permanental vector $G$ is infinitely divisible, then its kernel is diagonally equivalent to an inverse M-matrix. Therefore, if $G$ is an infinitely divisible kernel for a $1$-permanental vector, it should be diagonally equivalent to an inverse M-matrix.}.

We will now prove that any infinitely divisible kernel $G$ of a $1$-permanental vector has an equivalent kernel $\tilde{G}$ with a positive definite symmetric part. To see this, fix some infinitely divisible kernel $G$ of a $1$-permanental vector. By \cite[Lemma 4.2]{ek09}, it is diagonally equivalent to the inverse of a non-singular M-matrix $A$. Then, using \cite[Condition $I_{26}$]{plemmons77} for $A$, we know that there exists a diagonal matrix $D$ with positive diagonal terms, such that $D^{-1}AD$ has a positive definite symmetric part. Therefore, using Lemma \ref{lem:inv_psd} the inverse of $D^{-1}AD$ as $D^{-1}(A^{-1})D$ should also have a positive definite symmetric part. Since $A^{-1}$ is diagonally equivalent to the kernel $D^{-1}(A^{-1})D$, and $G$ is diagonally equivalent to $A^{-1}$, by transitivity of diagonal equivalence, $G$ should also be diagonally equivalent to the matrix $D^{-1}(A^{-1})D$, which has a positive definite symmetric part. \qed
\end{rem}

\vspace{.3em} \begin{rem}[Markovian kernels]
\label{rem:rem2-lem:density}
A sub-family of
kernels of infinitely divisible $1$-permanental vectors, which we call the Markovian kernels, is closely related to Markov chains and plays an important role in the study of local-time isomorphism theorems. Given an irreducible discrete-time Markov chain on the state space $[n]$, let us denote its transition kernel and stationary distribution by $P$ and $\pi$, respectively, where $\pi$ is thought of as a column vector. Let $\Pi=\diag(\pi)$ denote the diagonal matrix with diagonal terms given by $\pi$. The Laplacian associated with this Markov chain is defined as $Q=\Pi(P-I)$. Consider the continuous-time Markov process with rate matrix $Q$ (whose stationary distribution would then be the uniform distribution). We introduce a cemetery state $\Delta$, and a nonnegative killing rates given by $\bm{h}=(h_1, \cdots, h_n)\geq 0$.
Namely, when in state $i \in [n]$, there is now a rate $h_i$ of jumping to the cemetery state. The
rate matrix associated with
the corresponding continuous-time killed Markov process is defined to be $Q_{\bm{h}}=Q-\diag(\bm{h})$. It can be checked that $Q_{\bm{h}}$ is invertible when $\sum_i h_i > 0$, and we can define the Green's kernel associated with it as $G_{\bm{h}} : =-Q_{\bm{h}}^{-1}$. Any
matrix $G_{\bm{h}}$ obtained from such a construction is called a Markovian kernel.

We claim that $G_{\bm{h}}$ will have a positive definite symmetric part.
To see this, let $P^* := \Pi^{-1} P^T \Pi$
be the time reversal of $P$. Note that $P^*$ is a stochastic matrix. The additive-reversibilization of $P$ is defined to be the Markov chain with transition matrix $M := \frac{1}{2}(P + P^*)$, see
\cite[Lemma 9.20]{aldous-fill}.
Note that $\pi$ is the stationary distribution of $M$, since $\pi^T \frac{P + P^*}{2} = \frac{1}{2} \pi^T + \frac{1}{2} \pi^T \Pi^{-1} P^T \Pi = \pi^T$. Further, $M$ is reversible, since $\Pi M = \frac{1}{2} (\Pi P  + P^T \Pi) = M^T \Pi$.
This Markov chain is clearly irreducible, and its Laplacian is $\Pi (M-I)$,
which we denote by $S$. Since $S$ is the Laplacian of
a reversible Markov chain, $-S$ is nonnegative definite. This can be seen by observing that $S$ has the same eigenvalues as $\Pi^{\frac{1}{2}} S \Pi^{- \frac{1}{2}}$, which is symmetric, and so has real eigenvalues, while $S$, as the rate matrix of an irreducible continuous-time Markov chain, has $0$ as a simple eigenvalue and all other eigenvalues having strictly negative real parts.


By the Perron-Frobenius theorem, we know that $\langle v,-Sv\rangle=0$ if $v\in\textnormal{span}(\ind)$, and $\langle v,-Sv\rangle > 0$ otherwise. But if $v\in\textnormal{span}(\ind)$ and $v\neq 0$, then $v^T \diag(\bm{h}) v > 0$. Hence, for any $v$, we have $\langle v, (\diag(\bm{h})-S)v\rangle>0$ if $v\neq 0$, i.e. $\diag(\bm{h})-S$ is positive definite, i.e. $- \frac{1}{2}( Q_{\bm{h}} + Q_{\bm{h}}^T)$ is positive definite. By Lemma \ref{lem:inv_psd}, we know that $-Q_{\bm{h}}$ and hence $G_{\bm{h}}$ has a positive definite symmetric part.

It is apparent that $G_{\bm{h}}^{-1}=-Q_{\bm{h}}=\diag(\bm{h})+\Pi-\Pi  P$ has non-positive off-diagonal terms. We also showed that it has a positive definite symmetric part. Then we can use \cite[Condition $I_{26}$]{plemmons77} to see that $-Q_{\bm{h}}$ is indeed a non-singular M-matrix. Therefore, $G_{\bm{h}}$ is an inverse M-matrix and hence infinitely divisible by \cite[Lemma 4.2]{ek09}.
\qed

\end{rem}

\vspace{.3em} \begin{rem}[Positive definite kernels]
\label{rem:rem-psd-lem:density}

Let $G \in M_n(\mathbb{R})$ be a positive definite matrix. Let $\eta \sim N(0, G)$. One can see from the moment generating function that $\frac{1}{2} \eta^2$ is $\frac{1}{2}$-permanental with kernel $G$. Since the sum of two independent copies of a $\frac{1}{2}$-permanental random vector with a given kernel is a $1$-permanental vector with that kernel, it follows that for every positive definite matrix $G \in M_n(\mathbb{R})$, there is a $1$-permanental vector with kernel $G$. \qed
\end{rem}

{\color{blue}
}

\vspace{.3em} \begin{rem}
One might wonder if the condition
that $G$ should have a positive definite symmetric part is too stringent.
It is obvious that any positive definite and symmetric (PSD) matrix
satisfies this condition. Less obvious is that
infinitely divisible kernels also fall into this category up to diagonal equivalence
(see Remark \cref{rem-infdiv}).

In his classic 1967 paper, Vere-Jones~\cite{verejones97} gave necessary and sufficient conditions
on the kernel for the corresponding 1-permanental vector to exist. However,
these conditions
are not easy to check. These conditions are known to hold for matrices that are either diagonally equivalent to
a PSD matrix or are
infinitely divisible \cite{eisenbaum17}. We already showed in \cref{rem-infdiv} that any infinitely divisible kernel $G$ is diagonally equivalent to another kernel $\tilde{G}$ with a positive definite symmetric part. Therefore, both scenarios are included in Lemma \ref{lem:density} up to diagonal equivalence, in the sense that given any 1-permanental vector with such a kernel $G$, we can find an equivalent kernel $\tilde{G}$ that has a positive definite symmetric part and apply Lemma \ref{lem:density} to that kernel. We also remark that these two families of kernels are not inclusive of one another.

To the best of our knowledge, the only known instances of 1-permanental vectors whose kernel is neither diagonally equivalent to a PSD matrix nor an infinitely divisible one are the examples constructed in \cite[Theorem 7.3]{em17}, but their construction is very delicate, hence this family of kernels is rather restricted.
We remark that it is an important open problem to give an algorithmically checkable criterion that characterizes the set of 1-permanental kernels. \qed
\end{rem}

In the discussion following the proof of Lemma \ref{lem:density} below, and in the rest of the paper, we say that a function $f:\mathbb{C}^n\rightarrow \mathbb{C}$ has subexponential growth rate if $\lim_{\|\phi\|\rightarrow \infty}\frac{\log(1+|f(\phi)|)}{\|\phi\|}= 0$. We say that such a function is
sub-Gaussian if there exists some $c, c'>0$, such that $|f(\phi)|\leq ce^{-c'\|\phi\|^2}$.
Similar definitions apply to functions $f:\mathbb{R}_+^n\rightarrow \mathbb{C}$. We now proceed to prove Lemma \ref{lem:density}.

\begin{proof}
[Proof of Lemma \ref{lem:density}]
This proof mimics the procedure of Brydges-van der Hofstad-König~\cite{bhk07}. By the definition of a $1$-permanental vector, for any $\lambda\in\mathbb{R}_+^n$, we have the Laplace transform
    \[\E\big[e^{-\langle\lambda,\ell\rangle}\big]=\frac{1}{|I+\Lambda G|}=\frac{1}{|G||\Lambda -Q|}.\]

Here we used $I+\Lambda G=(G^{-1}+\Lambda)G=(\Lambda-Q)G$. Note that $|G|\neq 0$ by \cref{rem:rem1-lem:density}.

Under the conditions assumed in Lemma \ref{lem:density}, we can write $\frac{1}{2}(G + G^T) = - \frac{1}{2} (Q^{-1} + Q^{-T}) = - \frac{1}{2} Q^{-1} (Q + Q^T) Q^{-T}$, from which we see that, since $G$ is assumed to have a positive definite symmetric part, so does $- Q$, and hence so does $\Lambda-Q$. Now, consider $n$ complex numbers $\phi_i=u_i+\compi v_i, \,\forall i\in[n]$. Using the formula in \cite[Lemma 2.2]{bhk07} for $\int e^{\langle\phi,M\bar{\phi}\rangle}du\,dv$ for the matrix $M=\Lambda-Q$, we have
    \[\E\big[e^{-\langle\lambda,\ell\rangle}\big]
=\int_{\mathbb{R}^{2n}}\frac{1}{\pi^n|G|}\, e^{-\langle\phi, (\Lambda-Q)\bar{\phi}\rangle}du\,dv =\int_{\mathbb{R}^{2n}}\frac{1}{\pi^n|G|}\, e^{-\langle\phi, \Lambda\bar{\phi}\rangle+\langle \phi, Q\bar{\phi}\rangle}du\,dv.\]

Using the change-of-variable $\phi_i=\sqrt{l_i}\,e^{\compi\theta_i}$ where $l_i\in(0,\infty), \theta_i\in[0,2\pi)$
for all $i\in[n]$ and $\phi_i \neq 0$ gives
    \begin{equation*}
    \begin{aligned}
    \E\big[e^{-\langle\lambda,\ell\rangle}\big] &=\int_{\mathbb{R}_+^n}\int_{[0,2\pi)^n}\frac{1}{(2\pi)^n|G|}\, e^{-\langle\lambda, l\rangle+\sum_{i,j} Q_{ij}\sqrt{l_i}\sqrt{l_j}e^{\compi (\theta_i-\theta_j)}}d\theta\,dl\\
    &=\int_{\mathbb{R}_+^n} e^{-\langle\lambda, l\rangle} \bigg(\int_{[0,2\pi)^n}\frac{1}{(2\pi)^n|G|}\, e^{\sum_{i,j} Q_{ij}\sqrt{l_i}\sqrt{l_j}e^{\compi (\theta_i-\theta_j)}}d\theta\bigg)dl,
    \end{aligned}
    \end{equation*}

where we have used the fact that $du\, dv = 2^{-n}d\theta\,dl$. Let us define the measure $\rho(l)$ as
    \[\rho(l)=\int_{[0,2\pi)^n}\frac{1}{(2\pi)^n|G|}\, e^{\sum_{i,j} Q_{ij}\sqrt{l_i}\sqrt{l_j}e^{\compi (\theta_i-\theta_j)}}d\theta.\]

We claim that $\rho(l)$ is finite. To see this, by the triangle inequality, we have
    \begin{eqnarray}
    \label{eq:symmetrize}
|\rho(l)|
    &\leq& \int_{[0,2\pi)^n}\frac{1}{(2\pi)^n|G|}\, \Big|e^{\sum_{i,j} Q_{ij}\sqrt{l_i}\sqrt{l_j}e^{\compi (\theta_i-\theta_j)}}\Big| d\theta \nonumber\\
    &=& \int_{[0,2\pi)^n}\frac{1}{(2\pi)^n|G|}\, e^{\sum_{i,j} S_{ij}\sqrt{l_i}\sqrt{l_j}e^{\compi (\theta_i-\theta_j)}} d\theta,
    \end{eqnarray}

where $S:=\frac{1}{2}(Q+Q^T)$ denotes the symmetric part of $Q$.
Here, we used the fact that $|G|>0$ is necessary for the existence of a 1-permanental vector whose kernel has a positive definite symmetric part. (To see this, let $\Lambda=t I$ and take $t>0$ large enough; then the positivity of the Laplace transform $\E\big[e^{-\langle\lambda,\ell\rangle}\big]$ implies that $|G|>0$.) We have observed earlier that $-S$ has a positive definite symmetric part.
Let $\lambda_{\min}(-S)>0$ be the minimum eigenvalue of $-S$. Then we know by spectral bounds that
    \[|\rho(l)|
    \leq \int_{[0,2\pi)^n}\frac{1}{(2\pi)^n|G|}\, e^{-\lambda_{\min}(-S)\big(\sum_i l_i\big)} d\theta\\
= \frac{1}{|G|}\, e^{-\lambda_{\min}(-S)\big(\sum_i l_i\big)}<\infty.\]

This shows that $\rho(l)$ is indeed finite. Moreover, since
    \[\int_{\mathbb{R}_+^n}|\rho(l)|d\l\leq \int_{\mathbb{R}_+^n}\frac{1}{|G|}\, e^{-\lambda_{\min}(-S)\big(\sum_i l_i\big)}dl = \frac{1}{|G|(\lambda_{\min}(-S))^n}<\infty,\]

we have $\rho\in L^1(\mathbb{R}_+^n)$. Finally, to see that $\rho(l)$ is in fact real, note that
    \begin{equation*}
    \begin{aligned}
    \overline{\rho(l)}
    &=\int_{[0,2\pi)^n}\frac{1}{(2\pi)^n|G|}\, e^{\sum_{i,j} Q_{ij}\sqrt{l_i}\sqrt{l_j}e^{-\compi (\theta_i-\theta_j)}}d\theta\\
    &=\int_{[0,2\pi)^n}\frac{1}{(2\pi)^n|G|}\, e^{\sum_{i,j} Q_{ij}\sqrt{l_i}\sqrt{l_j}e^{\compi ((-\theta_i)-(-\theta_j))}}(-1)^n d(-\theta_1)\cdots d(-\theta_n)\\
    &=\int_{(-2\pi,0]^n}\frac{1}{(2\pi)^n|G|}\, e^{\sum_{i,j} Q_{ij}\sqrt{l_i}\sqrt{l_j}e^{\compi (\theta_i'-\theta_j')}} d\theta_1'\cdots d\theta_n'\\
    &=\int_{[0,2\pi)^n}\frac{1}{(2\pi)^n|G|}\, e^{\sum_{i,j} Q_{ij}\sqrt{l_i}\sqrt{l_j}e^{\compi (\theta_i'-\theta_j')}} d\theta_1'\cdots d\theta_n'\\
    &=\rho(l),
    \end{aligned}
    \end{equation*}

where in the third equality we used the substitution $\theta'=-\theta$. By
the invertibility of the Laplace transform given the region of convergence, $\rho(l)$ is exactly the density of a 1-permanental random vector with kernel $G$, if such a vector exists.
\end{proof}

From the proof above, we know that for any $G \in M_n(\mathbb{R})$ with a positive definite symmetric part,
$\rho(l)$ is a well-defined real density in $L^1(\mathbb{R}_+^n)$. Moreover, from its definition, we see that $\rho(l)$ is infinitely differentiable in the interior of $\mathbb{R}_+^n$.
\footnote{The integrand in \cref{eqn:densityformula}, after being differentiated for a finite number of times, would still be integrable, due to essentially the same reasoning as used in the proof of Lemma \ref{lem:density}.}
Therefore, the existence of a 1-permanental vector with kernel $G$ is equivalent to the nonnegativity
of the density $\rho(l)$.

We remark that Lemma \ref{lem:density} is
implicitly in
\cite{av20} and \cite{fr14}, but only in the Markovian setting.
To clarify this, we need to define the Markovian loop soups.

For a transient Markov process $\{X_t: t\geq 0\}$ on discrete state space $V$, and any $t>0$, let $\Omega_{r,t}$ be the space of right-continuous functions $[0, t] \rightarrow V$, with finitely many jumps and the same value at $0$ and $t$. Also, let $\Omega_r :=\cup_{t>0}\Omega_{r,t}$. We call this the space of rooted loops. For any $\phi\in \Omega_r$, let $\zeta(\phi)$ be the unique $t>0$ such that $\phi\in \Omega_{r,t}$.  We identify $\Omega_r$ with $\Omega_{r,1}\times (0,\infty)$ via the map $(w, t)\in \Omega_{r,1}\times (0, \infty) \rightarrow \phi(\cdot) = w(\frac{\cdot}{t})\in \Omega_r$, and endow $\Omega_{r,1}\times (0, \infty)$ with the canonical product $\sigma$-algebra (where $\Omega_{r,1}$ is endowed with the $\sigma$-algebra generated by the maps $X_s, 0 \leq s \leq 1$, from $\Omega_{r,1}$ into $V$). For any $x\in V$, define $P_{x,x}^t$ on $\Omega_{r,t}$ such that for any measurable subset $B_t\subset\Omega_{r,t}$,
\[P_{x,x}^t(B_t) = P(B_t\cap \{X_t = x\}\,|\,X_0=x).\]

We then define the rooted loop measure $\mu_r$ on $\Omega_r$ such that for any measurable subset $B\subset \Omega_r$,
\[\mu_r(B):=\sum_{x\in V}\int_0^\infty P_{x,x}^t(B)\frac{dt}{t}.\]
The shift operator $\theta_s$ acts on a rooted loop $\phi\in L_{r,t}$ by cyclically rotating the loop by time $s$. The set of unrooted loops, $\Omega_*$, is the
set of equivalence classes of rooted loops $\Omega_r$ under the equivalence relationship defined by the shifts $\{\theta_s: s>0\}$. The unrooted loop measure $\mu_*$ on $\Omega_*$ is then induced by the rooted loop measure under this equivalence relationship.

An unrooted loop soup (or Poisson gas of Markovian loops), at level $\alpha > 0$, is then a Poisson point process on the space $\Omega_*$ of unrooted loops, with intensity measure $\alpha \mu_*$. We will be mostly interested in loop soups at level $\alpha=1$. Therefore, we will always assume the level is at $1$ unless stated otherwise. Let $\Phi=\{\phi_i^*: i\in I\}$ be a countable collection of loops generated via such Poisson point process. Then we can define the loop soup occupation times (or occupation field) as $\ell^x(\Phi):=\sum_{i\in I}T^x(\phi_i)$ for any $x\in V$. Here $T^x(\phi)$ represents the local time the loop $\phi$ spends at state $x\in V$.

It is shown in Fitzsimmons and Rosen~\cite{fr14} that the Markovian loop soup local time $\ell$ is equal in distribution to a 1-permanental vector with the corresponding
Green's kernel, by using the generalized Le Jan's isomorphism theorem (non-symmetric version, Theorem 3.1 in \cite{fr14}).
Moreover, Theorem 2.15(a) in \cite{av20} states that for any bounded and continuous function $F:\mathbb{R}_+^n\rightarrow \mathbb{C}$, we have $\int \frac{1}{(2\pi \compi)^n|G|}F(\phi \bar{\phi})e^{\langle \phi,Q\bar{\phi}\rangle}d\bar{\phi} d\phi =\E F(\ell)$, where $\ell$ is the Markovian loop soup local time with intensity $\alpha=1$. This provides a connection between the
twisted Gaussian density and the Markovian loop soup local time: using the result by Fitzsimmons and Rosen~\cite{fr14}, we know that $\int \frac{1}{(2\pi \compi)^n|G|}F(\phi \bar{\phi})e^{\langle \phi, Q\bar{\phi}\rangle}d\bar{\phi} d\phi =\E F(\ell)$ for $\ell$ being a 1-permanental vector with a Markovian kernel $G$ as well. A polar coordinate change then reveals the density of the
1-permanental vector, in the Markovian case, in the form given in Lemma \ref{lem:density}.

In contrast, our density formula also works for non-Markovian kernels with a positive definite symmetric part,  and for this
we critically used the results derived by Brydges-Hofstad-König\cite{bhk07} in the proof above.

\vspace{.3em} \begin{rem}
A natural consequence of the above density formula is a symmetrization bound for the density.
Let $\iota:=\frac{|2G|}{|G+G^T|}$, and denote by $\tilde{\rho}(\bar{l})$
the density of a 1-permanental vector
$\bar{\ell}$ with the symmetric kernel $(-S)^{-1}$, where $S=\frac{1}{2}(Q+Q^T)$, with $Q := G^{-1}$. In the same setting as the previous lemma, it is immediate from its proof that
\[\rho(l)\leq \iota\bar{\rho}(l),\]
see equation \eqref{eq:symmetrize}.

This symmetrization bound will be used several times in the later sections. Similarly, one concludes that for any $\mathbb{C}$-valued function $f$ with a subexponential growth rate, we have $|\E[f(\ell)]|\leq \iota \E[|f(\bar{\ell})|]$:
\[|\E[f(\ell)]|=\bigg|\int_{\mathbb{R}_+^n} f(l)\rho(l) dl\bigg| \leq \int_{\mathbb{R}_+^n} |f(l)|\rho(l) dl \leq \iota \int_{\mathbb{R}_+^n} |f(l)|\bar{\rho}(l) dl = \iota\E[|f(\bar{\ell})|].\]

The integrand above is absolutely integrable for any function of subexponential growth since $\bar{\rho}(l)$ decays exponentially fast in
$\|l\|$. Moreover, the parameter $\iota$ that
appears in these bounds is always no less than 1. To see this, note
\[\int_{\mathbb{R}_+^n} \rho(l) dl\leq \int_{\mathbb{R}_+^n} \iota \rho(l) dl\quad\Rightarrow\quad 1\leq \iota.\]

In linear algebra language, for any matrix $G$ with a positive definite symmetric part,
we always have $|\frac{1}{2}(G+G^T)|\leq |G|$. This is a well-known fact called the Ostrowski-Taussky inequality in the literature~\cite{ot51}.

Another consequence of the proof above is the subexponentiality of $1$-permanental vectors. In the same setting as Lemma \ref{lem:density}, for the $1$-permanental vector $\ell$, we know that
\[\P(\ell\geq t)\leq \frac{1}{|G|}\int_{\prod_i[t_i,\infty)} e^{-\lambda_{\min}(-S)\|l\|_1}dl=\frac{1}{|G|(\lambda_{\min}(-S))^n}e^{-\lambda_{\min}(-S)\|t\|_1}.\]

Since for any $k\in\mathbb{Z}_+$, a $k$-permanental process with kernel $G$ is equal in distribution to $\sum_{i=1}^k\ell_i$, where $\ell_i$'s are i.i.d. 1-permanental vectors with kernel $G$, we also have
\[\P\bigg(\sum_{i=1}^k\ell_i\geq t\bigg)\leq k\, \P\Big(\ell\geq \frac{1}{k}t\Big)\leq c\, e^{-\frac{1}{k}\lambda_{\min}(-S)\|t\|_1}.\]

Here, $c>0$ does not depend on $t$. The subexponentiality of $k$-permanental vectors will also be used in later sections. \qed
\end{rem}

\subsection{Density of conditioned 1-permanental vectors}

The density formula allows us to study the density of a 1-permanental vector $\ell$ conditioned on some coordinate being equal to a fixed value $r$. Without loss of generality, we will consider $(\ell_2, \cdots, \ell_n\, |\, \ell_1=r)$. Let us also assume for now that $r>0$, the case $r=0$ being trivial, see \cref{rem:riszero} below.

To derive the conditional density $\rho_r(l_2, \cdots, l_n)$ of the conditioned process $(\ell_2, \cdots, \ell_n\, |\, \ell_1=r)$, let us decompose the Laplacian as
\begin{equation}
Q=\begin{pmatrix}Q_{11} & Q_{1*}^T\\
Q_{*1} & Q_{**}
\end{pmatrix}.
\label{eqn:decomp}
\end{equation}

Similarly, we
let $\phi_* :=(\phi_2,\cdots, \phi_n)$,
and, in general, whenever we use the underscript $*$, we mean the original vector or matrix restricted to $[n]\backslash\{1\}$. We then have the following density formula, for which the rotation symmetry in $\theta_i$'s plays an important role in the proof.

\vspace{.3em} \begin{lem}[Density of conditioned 1-permanental vector]
\label{lem:conditioned}
In the same setting as Lemma \ref{lem:density}, the conditional density $\rho_r(l_2, \cdots, l_n)$ of the conditioned process $(\ell_2, \cdots, \ell_n\, |\, \ell_1=r)$ is
\[\rho_r(l_2,\cdots, l_n) =  \int_{(0,2\pi)^{n-1}}\frac{|-Q_{**}|}{(2\pi)^{n-1}}\,
\exp\Big(\langle (\phi_*+ \sqrt{r}Q_{**}^{-T}Q_{1*}), Q_{**}(\bar{\phi}_*+ \sqrt{r}Q_{**}^{-1}Q_{*1})\rangle\Big)d\theta_*.\]
\end{lem}

\begin{proof}
First, plug in $l_1=r$ in the density formula for $\ell$. We know $\rho(r, l_2,\cdots, l_n)$ equals to
\[\int_{(0,2\pi)^n}\frac{e^{Q_{11}r}}{(2\pi)^n|G|} \exp\bigg(\sum_{i,j\neq 1} Q_{ij}\sqrt{l_i l_j}e^{\compi (\theta_i-\theta_j)}
+\sum_{i\neq 1}\big(Q_{i1}\sqrt{l_i r}e^{\compi(\theta_i-\theta_1)}+Q_{1i}\sqrt{l_i r}e^{\compi(\theta_1-\theta_i)}\big)
\bigg)d\theta.
\]

Now use the change of measure $\theta_i'=\theta_i-\theta_1, \forall i\neq 1$ and $\theta_1'=\theta_1$. We know it's further equal to
\[\int_{(0,2\pi)}\int_{(-\theta_1',-\theta_1'+2\pi)^{n-1}}\frac{e^{Q_{11}r}}{(2\pi)^n|G|}\, \exp\bigg(\sum_{i,j\neq 1} Q_{ij}\sqrt{l_i l_j}e^{\compi (\theta_i'-\theta_j')}
+\sum_{i\neq 1}\big(Q_{i1}\sqrt{l_i r}e^{\compi\theta_i'} +Q_{1i}\sqrt{l_i r}e^{-\compi\theta_i'}\big)
\bigg)d\theta_*'d\theta_1'.\]

But, given any $\theta_1'$, the inner integral is invariant whether it's integrated over $(-\theta_1',-\theta_1'+2\pi)^{n-1}$ or over $(0,2\pi)^{n-1}$, due to rotation symmetry. Hence we conclude that
\begin{equation*}
\begin{aligned}
\rho(r,l_2,& \cdots,l_n)\\
=& \int_{(0,2\pi)}\int_{(0,2\pi)^{n-1}}\frac{e^{Q_{11}r}}{(2\pi)^n|G|}\, \exp\bigg(\sum_{i,j\neq 1} Q_{ij}\sqrt{l_i l_j}e^{\compi (\theta_i'-\theta_j')}
+\sum_{i\neq 1}\big(Q_{i1}\sqrt{l_i r}e^{\compi \theta_i'}+Q_{1i}\sqrt{l_i r}e^{-\compi \theta_i'}\big)
\bigg)d\theta_*'d\theta_1'\\
=& \int_{(0,2\pi)^{n-1}}\frac{e^{Q_{11}r}}{(2\pi)^{n-1}|G|}\, \exp\bigg(\sum_{i,j\neq 1} Q_{ij}\sqrt{l_i l_j}e^{\compi (\theta_i'-\theta_j')}
+\sum_{i\neq 1}\big(Q_{i1}\sqrt{l_i r}e^{\compi \theta_i'}+Q_{1i}\sqrt{l_i r}e^{-\compi \theta_i'}\big)
\bigg)d\theta_*'.\\
\end{aligned}
\end{equation*}

Marginally, $\ell_1$ is distributed as an exponential random variable with mean $G_{11}$ as is shown in \cite{ek07}. Hence, we conclude that
\begin{equation*}
\begin{aligned}
\rho_r(l_2, & \cdots, l_n)
=G_{11}e^{G_{11}^{-1}r}\rho(r, l_2,\cdots, l_n) \\
=& \int_{(0,2\pi)^{n-1}}\frac{G_{11} e^{(Q_{11}+G_{11}^{-1})r}}{(2\pi)^{n-1}|G|}\, \exp\bigg(\sum_{i,j\neq 1} Q_{ij}\sqrt{l_i l_j}e^{\compi (\theta_i-\theta_j)}
+\sum_{i\neq 1}\big(Q_{i1}\sqrt{l_i r}e^{\compi \theta_i}+Q_{1i}\sqrt{l_i r}e^{-\compi \theta_i}\big)
\bigg)d\theta_*.\\
\end{aligned}
\end{equation*}

By the decomposition of the Laplacian $Q$, we have
\[\rho_r(l_2, \cdots, l_n)
=\int_{(0,2\pi)^{n-1}}\frac{G_{11} e^{(Q_{11}+G_{11}^{-1})r}}{(2\pi)^{n-1}|G|}\,
\exp\Big(\langle \phi_*, Q_{**}\bar{\phi}_*\rangle
+\sqrt{r}\big(\langle Q_{*1}, \phi_*\rangle+\langle Q_{1*}, \bar{\phi}_*\rangle\big)
\Big)d\theta_*.\]

We claim that $Q_{**}$ is invertible. To see this, note that by the adjugate formula $G=|-Q|^{-1}\textnormal{adj}(-Q)$, we have $G_{11}=\frac{|-Q_{**}|}{|-Q|}=|G||-Q_{**}|$. But $G_{11}>0$ and $|G|>0$ since $|G|$ is 1-permanental, so we conclude that $|-Q_{**}|>0$ and $Q_{**}$ is indeed invertible. Therefore we have
\begin{equation*}
\begin{aligned}
\rho_r(l_2, & \cdots, l_n)\\
=& \int_{(0,2\pi)^{n-1}}\frac{G_{11} e^{(Q_{11}+G_{11}^{-1})r}}{(2\pi)^{n-1}|G|}\,
\exp\bigg(\langle \phi_*, Q_{**}\bar{\phi}_*\rangle
+\big( \langle Q_{**}^T\phi_*, \sqrt{r}Q_{**}^{-1}Q_{*1}\rangle+\langle \sqrt{r}Q_{**}^{-T}Q_{1*}, Q_{**}\bar{\phi}_*\rangle\big)
\bigg)d\theta_*\\
=& \int_{(0,2\pi)^{n-1}}\frac{G_{11} e^{(Q_{11}+G_{11}^{-1}-Q_{1*}^T Q_{**}^{-1}Q_{*1})r}}{(2\pi)^{n-1}|G|}\,
\exp\bigg(\langle (\phi_*+ \sqrt{r}Q_{**}^{-T}Q_{1*}), Q_{**}(\bar{\phi}_*+ \sqrt{r}Q_{**}^{-1}Q_{*1})\rangle\bigg)d\theta_*.
\end{aligned}
\end{equation*}

Now, by the block matrix inverse formula, if $Q_{**}$ is invertible we have
\[G_{11}=(-Q_{11}+Q_{1*}^TQ_{**}^{-1}Q_{*1})^{-1}.\]

Therefore, $Q_{11}+G_{11}^{-1}-Q_{1*}^TQ_{**}^{-1}Q_{*1}=0$ and we conclude that
\[\rho_r(l_2, \cdots, l_n)=\int_{(0,2\pi)^{n-1}}\frac{|-Q_{**}|}{(2\pi)^{n-1}}\,
\exp\bigg(\langle (\phi_*+ \sqrt{r}Q_{**}^{-T}Q_{1*}), Q_{**}(\bar{\phi}_*+ \sqrt{r}Q_{**}^{-1}Q_{*1})\rangle\bigg)d\theta_*.\]

One can apply the same arguments as those in the proof of Lemma \ref{lem:density} to see that the integrand is absolutely integrable. Hence, the lemma is proved.
\end{proof}

\vspace{.3em} \begin{rem} \label{rem:riszero}
We remark that when $r=0$, we have
\[\rho_0(l_2, \cdots, l_n) = \int_{(0,2\pi)^{n-1}}\frac{|-Q_{**}|}{(2\pi)^{n-1}}\,
\exp\big(\langle \phi_*, Q_{**}\bar{\phi}_*\rangle\big)d\theta_*.\]

Computing the Laplace transform, we see that the conditioned 1-permanental process
\[
\{\ell_2,\cdots, \ell_n\,|\, \ell_1=0\}
\]
should again be a 1-permanental process with kernel $(-Q_{**})^{-1}$, which has positive definite symmetric part.

In the Markovian setting, this property has the following loop soup interpretation. A loop soup conditioned on not visiting some node $a$ is just an unconditioned loop soup on the loops that do not visit $a$, by the independence property of the Poissonian point process. This observation is mentioned in several previous works on loop soup (for example, \cite{sznitman12, fr14}). However, it is not a priori obvious that every 1-permanental $G$ with a positive definite symmetric part
will also satisfy a similar property, which is what we have established.
\qed
\end{rem}

Another interesting observation is that when $G$ is positive definite
we have the following conditional law for squared Gaussian vectors.
\vspace{.3em} \begin{cor}[Conditioned $\chi$-squared process]
\label{cor:chi-square}
For any non-degenerate Gaussian vector $\tilde{U}\sim\mathcal{N}(0,G)$, let $U\sim\mathcal{N}(-\sqrt{r}Q_{**}^{-1}Q_{*1},(-Q_{**})^{-1})$. We have
\[(\tilde{U}_*^2\,|\,\tilde{U}_1^2=r)\,\stackrel{d.}{=}\,U^2.\]
\end{cor}

\begin{proof}
By Lemma \ref{lem:conditioned}, in the Gaussian case, we have
\begin{equation*}
\begin{aligned}
\rho_r(l_2, & \cdots, l_n)\\
=& \int_{(0,2\pi)^{n-1}}\frac{|-Q_{**}|}{(2\pi)^{n-1}}\,
\exp\bigg(\big[ \big((u_*+ \sqrt{r}Q_{**}^{-1}Q_{*1})+\compi v_*\big), Q_{**}\big((u_*+ \sqrt{r}Q_{**}^{-1}Q_{*1})-\compi v_*\big)\big]_\mu\bigg)d\theta_*\\
=& \int_{(0,2\pi)^{n-1}}\frac{|-Q_{**}|}{(2\pi)^{n-1}}\,
\exp\bigg(\big[ (u_*+ \sqrt{r}Q_{**}^{-1}Q_{*1}), Q_{**}(u_*+ \sqrt{r}Q_{**}^{-1}Q_{*1})\big]_\mu+\langle v_*, Q_{**}v_*\rangle\bigg)d\theta_*.
\end{aligned}
\end{equation*}

This is exactly the law of $\big\{(U-\sqrt{r}Q_{**}^{-1}Q_{*1})^2+V^2\big\}$, for $U, V$ being independent $\mathcal{N}(0,\frac{1}{2}(-Q_{**})^{-1})$ vectors. Since we know that $\ell\stackrel{d.}{=} \{\tilde{U}^2+\tilde{V}^2:\,i\in[n]\}$, where $\tilde{U}, \tilde{V}$ are independent $\mathcal{N}(0,\frac{1}{2}G)$ vectors, we have

\[\{\ell_*\,|\,\ell_1=r\}\stackrel{d.}{=} \{\tilde{U}_*^2+\tilde{V}_*^2\,|\, \tilde{U}_1^2+\tilde{V}_1^2=r\}\stackrel{d.}{=} \big\{(U-\sqrt{r}Q_{**}^{-1}Q_{*1})^2+V^2\big\}.\]

We claim that
\[\{\tilde{U}_*^2+\tilde{V}_*^2\,|\, \tilde{U}_1^2+\tilde{V}_1^2=r\}\stackrel{d.}{=}\{\tilde{U}_*^2+\tilde{V}_*^2\,|\, \tilde{U}_1^2=r,\tilde{V}_1^2=0\}.\]

To see this, it suffices to show that, conditioned on $\ell_1=r$,  $\ell_*$ is independent of $\theta_1$, and therefore independent of $(U_1, V_1)$. This is readily concluded from the proof of Lemma \ref{lem:conditioned}. Indeed, fixing any $\theta_1$, the inner integral there does not depend on $\theta_1$. Now note that $\{\tilde{V}_*^2\,|\,\tilde{V}_1^2=0\}\stackrel{d.}{=}V^2$. It is immediate that
\[\{\tilde{U}_*^2\,|\, \tilde{U}_1^2=r\}\stackrel{d.}{=}\big\{(U-\sqrt{r}Q_{**}^{-1}Q_{*1})^2\big\}.\]
\end{proof}

A special version of this corollary, applying to a certain family of Gaussian free fields, was used earlier in studying isomorphism theorems~\cite{sznitman12}. Proving this theorem by directly working with the $\chi$-squared process seems to be hard. Here, the rotation symmetry trick (i.e., Lemma \ref{lem:conditioned}) helps us circumvent this issue.

\section{Non-reversible isomorphism theorems}

In the Markovian setting, the 1-permanental process is equal in distribution to the loop soup local time when the intensity of the soup is given by the loop measure (instead of $1/2$ of the loop measure). In light of this, the following theorem can be seen as a generalization of Le Jan's isomorphism theorem that connects the loop soup local time with a certain Gaussian integral in the scenario when the intensity
is $1/2$ of the loop measure (see, e.g., Theorem 4.5 in \cite{sznitman12}).

\vspace{.3em} \begin{thm}[Non-reversible Le Jan's isomorphism theorem]
\label{thm:lejan}
In the same setting as Lemma \ref{lem:density}, for any real function on $\mathbb{R}_+^n$ with a subexponential growth rate, we have
\begin{equation}
\E[f(\ell)] = \int_{\mathbb{C}^n} \frac{1}{(2\pi\compi)^n|G|}\,f(\phi\bar{\phi})\, e^{\langle \phi, Q\bar{\phi}\rangle} d\bar{\phi} d\phi.
\label{eqn:lejannr}
\end{equation}

Here $\phi\bar{\phi}:=(\phi_i\bar{\phi}_i: i\in V)$. Moreover, for any $k\in\mathbb{N}_+$, the $k$-permanental vector $\ell^{(k)}$ satisfies
\[\E\big[f(\ell^{(k)})\big]=\int_{\mathbb{C}^{n\times k}}\frac{1}{(2\pi\compi)^{nk}|G|^k}f\big(\phi^{(1)}\overline{\phi^{(1)}}+\cdots+\phi^{(k)}\overline{\phi^{(k)}}\big)e^{\sum_{i=1}^k\langle \phi^{(i)}, Q\overline{\phi^{(i)}}\rangle} \prod_i d\overline{\phi^{(i)}} d\phi^{(i)}.\]
\end{thm}

\begin{proof}
We know that
\begin{equation*}
\begin{aligned}
\E[f(\ell)]
&= \int_{\mathbb{R}_+^n} \int_{(0,2\pi)^n}\frac{1}{(2\pi)^n|G|}\, f(l)\, e^{\sum_{i,j} Q_{ij}\sqrt{l_i}\sqrt{l_j}e^{\compi (\theta_i-\theta_j)}}d\theta dl \\
&= \int_{\mathbb{C}^n} \frac{1}{(2\pi\compi)^n|G|}\, f(\phi\bar{\phi})\, e^{\langle \phi, Q\bar{\phi}\rangle} d\bar{\phi} d\phi . \\
\end{aligned}
\end{equation*}

Here we used change of variable $\phi_i=\sqrt{l_i}e^{\compi \theta_i}$. Now, using the formula repeatedly and the fact that $\ell^{(k)}$ is distributed as the sum of $k$ independent 1-permanental vectors proves the theorem.
\end{proof}

Using the distributional equivalence between a $1$-permanental vector and loop soup local time when the intensity of the soup is given by the loop measure, the probabilistic content of \cref{thm:lejan} can be equivalently stated in terms of loop soup local times. When the Markov chain is reversible, the right-hand side of equation \eqref{eqn:lejannr} becomes exactly complex Gaussian density. To see this connection, observe that we have $Q = Q^\top$ in the reversible case, and so
\( (\pi)^{ - n} e^{\langle \phi, Q\bar{\phi}\rangle} \)
can be written as a product of two Gaussians, as
\[
\frac{1}{(\pi)^{\frac{n}{2}}} e^{\langle u, u \rangle}
\frac{1}{(\pi)^{\frac{n}{2}}} e^{\langle v, v \rangle},
\]
where $\phi = u + \compi v$. The theorem therefore implies that $\phi\bar{\phi}$ where $\phi\sim\mathcal{CN}(0,G)$\footnote{It is standard notation for a
complex Gaussian that $\phi:=u+\compi v\sim\mathcal{CN}(0,\Gamma)$ iff \[\begin{pmatrix}
u\\ v
\end{pmatrix}\sim\mathcal{N}\bigg(\begin{pmatrix}
0\\ 0
\end{pmatrix},\frac{1}{2}\begin{pmatrix}
    \Re(\Gamma) & -\Im(\Gamma) \\
-\Im(\Gamma) & \Re(\Gamma)
\end{pmatrix}\bigg).\]} and $\ell$, the $1$-permanental vector, are equal in distribution in this case. This special case is now seen to be a consequence of the reversible Le Jan's isomorphism theorem, i.e., Theorem 4.5 in~\cite{sznitman12}. It is in this sense that
\cref{thm:lejan} can be seen as a generalization
to the non-reversible case of Theorem 4.5 in~\cite{sznitman12}.


In the
rest of this section, we are going to give a unified proof of several non-reversible isomorphism theorems. There will be two ingredients important in our proof: one is a pair of Ward identities for twisted Gaussian measure along the lines of those in \cite{bhs21}, and the other is the previously derived density formula in equation \eqref{eqn:densityformula}. In this way, we obtain non-reversible versions of Dynkin's isomorphism theorem~\cite{eisenbaum13, lejan08}, the generalized second Ray-Knight isomorphism theorem~\cite{ek09}, and Eisenbaum's isomorphism theorem (which is new).

\subsection{Ward identities for non-reversible Markov chains}

Let $P$ denote an irreducible transition probability matrix for a Markov chain on the state space $[n]$, and let the stationary distribution be denoted by $\pi$. Let $Q=\Pi(P-I)$ denote the normalized Laplacian, where $\Pi$ denotes the diagonal matrix with diagonal entries given by $\pi$. We think of $Q$ as the rate matrix of $\{X_s:\, s\geq 0\}$, a continuous-time irreducible Markov chain on the finite state space $[n]$, starting at $X_0=a$. Define the local time of the Markov chain as
\begin{equation}
\Ell_t^i=\int_0^t \ind\{X_s=i\}ds.
\label{eqn:localtime}
\end{equation}

For any $r\geq 0$, let the inverse local time be $\inv(r):=\inf\{t : \Ell_t^a>r\}$. We will consider the following unnormalized (since $-Q$ is only positive semi-definite here) twisted Gaussian measure
\[d\mu(\phi)= \pi^{-n} e^{\langle \phi, Q\bar{\phi}\rangle}dudv=(2\pi\compi)^{-n}e^{\langle \phi, Q\bar{\phi}\rangle}d\bar{\phi} d\phi .\]

Although the measure is unnormalized, when $g(\phi):\mathbb{C}^n\rightarrow \mathbb{C}$ is sub-Gaussian, it is still integrable with respect to the measure $d\mu(\phi)$. To see this, note that
\[\big|[ g(\phi)]_\mu\big|
\leq \int_{\orth} \pi^{-n} e^{-c(\|u\|^2+\|v\|^2)}e^{\langle \phi, S\bar{\phi}\rangle}dudv
\leq \int_\orth \pi^{-n} e^{-c(\|u\|^2+\|v\|^2)}dudv<\infty.\]

Here $[\cdot]_\mu$ denotes integration with respect to $d\mu(\phi)$
and
\begin{equation}    \label{eq:defineS}
S=\frac{1}{2}(Q+Q^T)
\end{equation}
denotes the symmetric part of $Q$, which we note satisfies $S\preceq 0$ in the Loewner partial order.

A function $f:\mathbb{C}^n\rightarrow \mathbb{C}$ or $f:\mathbb{R}_+^n\rightarrow \mathbb{C}$ is said to have sub-polynomial tails if for every $d>0$, there exists some constant $c_d>0$ such that $|f(l)|\leq c_d \|l\|^{-d}$. We will dub the functions in $C^\infty(\mathbb{R}_+^n)$ or $C^\infty(\mathbb{C}^n)$ with sub-polynomial tails for all orders of derivatives as having property $\mathfrak{F}$. Similarly we say that a function $g:\mathbb{C}^n\rightarrow \mathbb{C}$ or $g:\mathbb{R}_+^n\rightarrow \mathbb{C}$ has polynomial growth if there exist some $d > 0$, $b > 0$, and $c\in (0,\infty)$ such that $|g(\phi)|\leq b \|\phi\|^{d}$ for all $\phi$ that is large enough, i.e., $\|\phi\|>c$. We say that $g(\phi)$ in $C^\infty(\mathbb{R}_+^n)$ or $C^\infty(\mathbb{C}^n)$ has property $\mathfrak{G}$ if for any $k\in\mathbb{N}$, there exists some $d_k>0$, $b_k>0$ and $c_k \in (0,\infty)$  such that $|g^{(k)}(\phi)|\leq b_k \|\phi\|^{d_k}$ for all $\|\phi\|>c_k$, i.e. if every derivative of $g$ has polynomial growth.

In the scenario where one has a killing rate $h>0$ at each node, the Laplacian is replaced by $Q_h:=Q-hI$. One can then consider the Green's function $G_h:=(hI-Q)^{-1}$ since $Q_h$ is invertible. Note that there will be a 1-permanental random variable with such a kernel~\cite{ek09}.

Define the differential operator $T_i$ as $T_ig(\phi)=\frac{\partial }{\partial \bar{\phi}_i}g(\phi)$. Similarly, define $P_i$ as $P_ig(\phi)=\frac{\partial}{\partial \phi_i}g(\phi)$.\footnote{We can take $\frac{\partial}{\partial \phi_i}:=\frac{1}{2} \left( \frac{\partial}{\partial u_i}-\compi \frac{\partial}{\partial v_i}\right)$ and $\frac{\partial}{\partial \bar{\phi}}:=
\frac{1}{2} \left( \frac{\partial}{\partial u_i}+\compi \frac{\partial}{\partial v_i} \right)$.} We also define $\partial_i f(l):=\frac{\partial}{\partial l_i}f(l)$ as differentiation in $l_i\in\mathbb{R}_+$. We use $\E_{i,l}$ to denote the expectation with respect to a Markov chain starting at state $i$ with initial local time $l$. Similarly, let $\E_{i,l}^*$ denote expectation for the reversed chain starting at state $i$ with initial local time $l$. We say that a function $f(i,l):\mathcal{X}\times \mathbb{R}_+^n\rightarrow \mathbb{C}$ has property $\mathfrak{F}$ if it has property $\mathfrak{F}$ in $l$ for each fixed first component $i\in\mathcal{X}$. We then have the following Ward identities.

\vspace{.3em} \begin{lem}[Ward identities for twisted Gaussian measure]
\label{lem:ward}
Consider an irreducible continuous-time Markov chain on discrete space $[n]$ with Laplacian $Q$. For $f(i,l):\mathcal{X}\times \mathbb{R}_+^n\mapsto \mathbb{C}$ with property $\mathfrak{F}$ and $g(\phi):\mathbb{C}^n\rightarrow\mathbb{C}$ with property $\mathfrak{G}$, we have
\[\sum_i[ \phi_i g(\phi) f(i,\phi\bar{\phi}) ]_\mu=\sum_i\bigg[ (T_ig)(\phi)\int_0^\infty \E_{i,\phi\bar{\phi}}f(X_t, \Ell_t)dt\bigg]_\mu,\]
and
\[\sum_i[ \bar{\phi}_i g(\phi) f(i,\phi\bar{\phi})]_\mu=\sum_i\bigg[ (P_ig)(\phi)\int_0^\infty \E_{i,\phi\bar{\phi}}^*f(X_t, \Ell_t)dt\bigg]_\mu.\]
\end{lem}

\begin{proof}
First note that\footnote{Note that the integrals on the right-hand side exist because of absolute integrability, since property $\mathfrak{F}$ of the function $f$ and property $\mathfrak{G}$ of the function $g$ guarantees that given any $d\in\mathbb{N}$, terms like $|(T_ig)(\phi)f(i,\phi\bar{\phi})|$ will be bounded by $r_d(1+\|\phi\|)^{-d}$ for some constant $r_d>0$.}
\begin{equation*}
\begin{aligned}
0
&=\int_\orth (2 \pi \compi)^{-n} T_i(g(\phi)f(i,\phi\bar{\phi})e^{\langle \phi,Q\bar{\phi}\rangle})d \bar{\phi} d \phi\\
&= [ (T_ig)(\phi)f(i,\phi\bar{\phi})]_\mu
+[ \phi_i g(\phi)\partial_if(i,\phi\bar{\phi})]_\mu
+[ g(\phi)(Q^T\phi)_if(i,\phi\bar{\phi})]_\mu.
\end{aligned}
\end{equation*}

Summing over all $i\in[n]$. we have
\begin{equation*}
\begin{aligned}
\sum_i[ (T_ig)(\phi)f(i,\phi\bar{\phi})]_\mu
&=
-\sum_i[ \phi_i g(\phi)\partial_if(i,\phi\bar{\phi})]_\mu
-\sum_i[ g(\phi)(Q^T\phi)_if(i,\phi\bar{\phi})]_\mu\\
&= -\sum_i[ \phi_i g(\phi)\partial_if(i,\phi\bar{\phi})]_\mu
-\sum_i[ \phi_ig(\phi)(Qf)(i,\phi\bar{\phi})]_\mu.
\end{aligned}
\end{equation*}

Let $\mathcal{Q}$ be the generator of the joint process $(X_t, \Ell_t)$, which is a Markov process itself. It's not hard to see that we have $(\mathcal{Q} f)(i,l)=(\partial_if+Qf)(i,l)$, and hence
\begin{equation*}
\begin{aligned}
\sum_i[(T_ig)(\phi)f(i,\phi\bar{\phi})]_\mu
&= -\sum_i[\phi_i g(\phi)\,(\partial_if
+Qf)(i,\phi\bar{\phi}))]_\mu\\
&= -\sum_i[\phi_i g(\phi)\,(\mathcal{Q} f)(i,\phi\bar{\phi})]_\mu.
\end{aligned}
\end{equation*}

Now we can use Komogorov's backward equation for the joint process $(X_t,\Ell_t)$, and deduce that $\mathcal{Q} f_t=\partial_t f_t$ for $f_t(i,l):=\E_{i,l}(f(X_t, \Ell_t))$. Hence letting $f=f_t$, we have
\begin{equation*}
\begin{aligned}
\int_0^\infty\sum_i [(T_ig)(\phi) & f_t(i,\phi\bar{\phi})]_\mu dt
= - \int_0^\infty \sum_i[\phi_i g(\phi)\,(\mathcal{Q} f_t)(i,\phi\bar{\phi})]_\mu dt\\
&= - \int_0^\infty \sum_i[\phi_i g(\phi)\,(\partial_t f_t)(i,\phi\bar{\phi})]_\mu dt= - \sum_i[\phi_i g(\phi)\,f_t(i,\phi\bar{\phi})]_\mu \Big|_0^\infty.
\end{aligned}
\end{equation*}

Now we note that $\phi_ig(\phi)\E_{i,\phi\bar{\phi}}(f(X_t,\Ell_t))\rightarrow 0$ almost surely. Since the event $\{\forall i\in\mathcal{X},\,\Ell_t^i \rightarrow \infty\}$ happens almost surely, we know the r.h.s. is further equal to $\sum_i[\phi_i g(\phi)\,f(i,\phi\bar{\phi})]_\mu$. On the other hand, we know the l.h.s. is indeed $\sum_i\big[(T_ig)(\phi)\int_0^\infty \E_{i,\phi\bar{\phi}}f(X_t, \Ell_t)dt\big]_\mu$ by absolute integrability. The detailed proof of absolute integrability is deferred to the appendix (\cref{apx:sec1}).

If one starts with the following equation:
\begin{equation*}
\begin{aligned}
0
&=\int_\orth (2 \pi \compi)^{-n} P_i(g(\phi)f(i,\phi\bar{\phi})e^{\langle \phi,Q\bar{\phi}\rangle})d \bar{\phi} d \phi \\
&= [(P_ig)(\phi)f(i,\phi\bar{\phi})]_\mu
+[\bar{\phi}_i g(\phi)\partial_if(i,\phi\bar{\phi})]_\mu
+[g(\phi)(Q\bar{\phi})_if(i,\phi\bar{\phi})]_\mu,
\end{aligned}
\end{equation*}

then one would arrive at the second equation. Note that the reversed chain will be irreducible if the original chain is.
\end{proof}

\subsection{Non-reversible Dynkin's isomorphism theorem}

We have the following theorem, previously proved by Eisenbaum\cite{eisenbaum13} and Le Jan\cite{lejan08}, extending Dynkin's isomorphism theorem from the reversible case\cite{sznitman12}. Here $\E_a$ is the expectation operator associated with the Markov chain started at state $a\in\mathcal{X}$.

\vspace{.3em} \begin{thm}[non-reversible Dynkin's isomorphism theorem]
\label{thm:dynkin}
Consider an irreducible continuous-time Markov chain on discrete space $[n]$ with Laplacian $Q$ and killing rate $\bm{h}=(h_1, \cdots, h_n)\geq 0$ such that $\sum_i h_i>0$. Let $\ell$ be a 1-permanental process with kernel $G_{\bm{h}}=(\diag(h)-Q)^{-1}$. For $ u:\orthp\rightarrow \mathbb{C}$ with property $\mathfrak{F}$, we have for any $a\in[n]$,
\[\E\,[\ell_a u(\ell)]=\int_0^\infty \E_a u(\Ell_t+\ell)\ind_{X_t=a}e^{-\sum_i h_i\Ell_t^i}dt=\int_0^\infty \E_a^* u(\Ell_t+\ell)\ind_{X_t=a}e^{-\sum_i h_i\Ell_t^i}dt. \]

\end{thm}

We remark that in the equation above, the operator $\E_a$ also takes the expectation with respect to the 1-permanental process $\ell$. Note that $G_{\bm{h}}:=(\diag(\bm{h})-Q)^{-1}$ is positive definite for any $\bm{h}\geq 0$, s.t. $\sum_i h_i>0$. Hence the 1-permanental vector indeed exists \cite{ek09}. According to \cref{thm:lejan}, for any real function on $\mathbb{R}_+^n$ with a subexponential growth rate, we have
\[\E[ u(\ell)] = [u(\phi\bar{\phi})]_{\mu_{\bm{h}}}.\]

Here $[\cdot]_{\mu_{\bm{h}}}$ denotes integration with respect to the measure $d\mu_{\bm{h}}(\phi):=\frac{1}{(2\pi\compi)^n|G_h|}e^{\langle \phi, Q_{\bm{h}}\bar{\phi}\rangle}d\bar{\phi}d\phi$, where $Q_{\bm{h}}:=Q-\diag(\bm{h})$. Now we prove \cref{thm:dynkin} above.
\begin{proof}

Take $g(\phi)=\bar{\phi}_b$, and $f(i,l):= v(l)\ind_{i=a}$ for some function $v:\mathbb{R}_+^n\rightarrow \mathbb{C}$ that will be determined later in the first equation of Lemma \ref{lem:ward}, we obtain
\[\sum_i[ \phi_i \bar{\phi}_b v(\phi\bar{\phi})\ind_{i=a} ]_\mu=\sum_i\bigg[ (T_ig)(\phi)\int_0^\infty \E_{i,\phi\bar{\phi}}v(\Ell_t)\ind_{X_t=a}dt\bigg]_\mu.\]

Note that $T_ig(\phi)=1$ if $i=b$, and $T_ig(\phi)=0$ otherwise. Also note $\E_{b,\phi\bar{\phi}}v(\Ell_t)\ind_{X_t=a}=\E_b v(\Ell_t+\phi\bar{\phi})\ind_{X_t=a}$. Simplifying the equation above, we have
\[[\phi_a\bar{\phi}_b v(\phi\bar{\phi})]_\mu=\bigg[ \int_0^\infty \E_b v(\Ell_t+\phi\bar{\phi})\ind_{X_t=a}dt\bigg]_\mu. \]

Now let $ v(l):=u(l)e^{-\sum_ih_il_i}$ in the equation above and get
\[[\phi_a\bar{\phi}_b u(\phi\bar{\phi})e^{-\sum_i h_i|\phi_i|^2}]_\mu=\bigg[\bigg(\int_0^\infty \E_b u(\Ell_t+\phi\bar{\phi})\ind_{X_t=a}e^{-\sum_i h_i\Ell_t^i}dt\bigg)e^{-\sum_i h_i|\phi_i|^2}\bigg]_\mu. \]

Equivalently, after dividing both sides by $|G_{\bm{h}}|$, we have
\[[\phi_a\bar{\phi}_b u(\phi\bar{\phi})]_{\mu_{\bm{h}}}=\bigg[\int_0^\infty \E_b u(\Ell_t+\phi\bar{\phi})\ind_{X_t=a}e^{-\sum_i h_i\Ell_t^i}dt\bigg]_{\mu_{\bm{h}}}.\]

We remark that the measure $\mu_{\bm{h}}$ is normalizable in this case since $Q_{\bm{h}}$ has a positive definite symmetric part.
When $a=b$, we can use \cref{thm:lejan} on both sides of the integration, and obtain the first equality in the theorem due to the following equations.
\begin{align*}
& [\phi_a\bar{\phi}_a u(\phi\bar{\phi})]_{\mu_{\bm{h}}}=\E\,[\ell_a u(\ell)],\\
& \bigg[\int_0^\infty \E_b u(\Ell_t+\phi\bar{\phi})\ind_{X_t=a}e^{-\sum_i h_i\Ell_t^i}dt\bigg]_{\mu_{\bm{h}}}=\int_0^\infty \E_a u(\Ell_t+\ell)\ind_{X_t=a}e^{-\sum_i h_i\Ell_t^i}dt.
\end{align*}

For the second equality, we use the second equation in Lemma \ref{lem:ward}. The rest of the computation is essentially the same as above.
\end{proof}

The general case, when $a,b\in[n]$ are not necessarily equal, corresponds to the isomorphism theorem as stated in Proposition 1 in \cite{lejan08}. However, when $a\neq b$, and $Q$ is non-reversible, the l.h.s does not have a direct probabilistic meaning.

\subsection{non-reversible Ray-Knight's isomorphism}

The following theorem, first proved by Eisenbaum-Kaspi\cite{ek09}, extends the generalized second Ray-Knight's isomorphism theorem to the non-reversible case (the reversible version of generalized Ray-Knight isomorphism theorem can be found in, e.g., \cite{sznitman12}).

\vspace{.3em} \begin{thm}[non-reversible generalized second Ray-Knight's isomorphism theorem]
\label{thm:rayknight}
Given a Laplacian $Q$ for an irreducible finite Markov chain, let $\ell$ be a 1-permanental process with kernel $G^h=(he_ae_a^T-Q)^{-1}$ for some $h>0$ and some $a\in[n]$. For $ u:\orthp\rightarrow \mathbb{C}$ a smooth compactly supported function and any $r\geq 0$,
\[\E_a\big[ u(\ell+\Ell_{\inv(r)})\,|\,\ell_a=0\big]=\E\big[ u(\ell)\,|\,\ell_a=r\big].\]
Here $\inv(r):=\inf\{s\geq 0: \Ell_s^a\geq r\}$.
\end{thm}

To see that the symmetric part of $G^h$ is positive definite, refer to the discussion in \cref{rem:rem2-lem:density}.

\begin{proof} Directly using the choice of $g$ similar to that in Bauerschmidt's proof\cite{bhs21} in the Ward identities will incur some problems, because we are dealing with complex numbers. For any $\epsilon>0$, let $\delta_{0,\epsilon}(u)$ be a smooth approximation of the delta function with support on $\{u\in \mathbb{R}:\,|u|< \epsilon\}$.

The case $r=0$ is trivial since under $\P_a$, we have $\Ell_{\inv(0)}=0$ and the theorem holds. Hence, we consider from now on $r>0$, and any $\epsilon<\frac{r}{4}$. In the first equation of Lemma \ref{lem:ward}, we take $g_\epsilon(\phi)=\int_{2\epsilon}^{r-2\epsilon}\frac{1}{\phi_a}\delta_{0,\epsilon}(\phi_a\bar{\phi}_a-s)ds$. Note this equation is meaningless when $\phi_a=0$, so we simply define $g_\epsilon(\phi)=0$ when $\phi_a=0$. We also take $f(i,l)=  u(l)\delta_{0,\epsilon}(l_a-r)\ind_{i=a}$. In this case, we have the integration formula\footnote{We remark that integral for both sides exists since the integrand inside each bracket is continuous with bounded support.}
\[[\phi_a g_\epsilon(\phi)  u(\phi\bar{\phi})\delta_{0,\epsilon}(\phi_a\bar{\phi}_a-r)]_\mu=\bigg[(T_ag_\epsilon)(\phi)\int_0^\infty \E_a u(\phi\bar{\phi}+\Ell_t)\delta_{0,\epsilon}( \phi_a\bar{\phi}_a+\Ell_t^a-r)\ind_{X_t=a}dt\bigg]_\mu.\]

Note that $g_\epsilon(\phi)$ is nonzero only if $\phi_a\bar{\phi}_a< r-\epsilon$, which is disjoint from the support of $\delta_{0,\epsilon}(\phi_a\bar{\phi}_a-r)$ since it's non-zero only if $\phi_a\bar{\phi}_a> r-\epsilon$. Hence, the left-hand side is always zero.

Now consider the right-hand side. We have
\[T_ag_\epsilon(\phi)=\int_{2\epsilon}^{r-2\epsilon}\frac{1}{\phi_a}T_a\delta_{0,\epsilon}(\phi_a\bar{\phi}_a-s)ds=\int_{2\epsilon}^{r-2\epsilon}\delta_{0,\epsilon}'(\phi_a\bar{\phi}_a-s)ds.\]

Here the validity of the interchange of integration and differentiation is due to the continuity of $g_\epsilon(\phi)$ and $T_a g_\epsilon(\phi)$, which satisfies the Leibniz rule. So in $(l,\theta)$-coordinates, we have
\[T_ag_\epsilon(\phi)=\int_{2\epsilon}^{r-2\epsilon}\delta_{0,\epsilon}'(l_a-s)ds=\delta_{0,\epsilon}(l_a-2\epsilon)-\delta_{0,\epsilon}(l_a-r+2\epsilon).\]

Therefore, we conclude for any sufficiently small $\epsilon>0$,
\[\begin{aligned}
\bigg[\delta_{0,\epsilon}(l_a-2\epsilon)& \int_0^\infty \E_a u(l+\Ell_t)\delta_{0,\epsilon}( l_a+\Ell_t^a-r)\ind_{X_t=a}dt\bigg]_\mu\\
&=\bigg[\delta_{0,\epsilon}(l_a-r+2\epsilon)\int_0^\infty \E_a u(l+\Ell_t)\delta_{0,\epsilon}( l_a+\Ell_t^a-r)\ind_{X_t=a}dt\bigg]_\mu.\end{aligned}\]

Fix any $l\in\orthp$ s.t. $l_a\leq r$. Let $dL^a = \ind_{X_t=a}dt$, we have
\begin{equation*}
\begin{aligned}
\int_0^\infty\E_a & u(l+\Ell_t) \delta_{0,\epsilon}( l_a+\Ell_t^a-r)\ind_{X_t=a}dt\\
&=\E_a \int_0^\infty u(l+\Ell_t) \delta_{0,\epsilon}( l_a+\Ell_t^a-r)\ind_{X_t=a}dt\\
&=\E_a \int_0^\infty u(l+\Ell_{\inv(L^a)})\delta_{0,\epsilon}( l_a+L^a-r)dL^a\\
&=\int_0^\infty\E_a u(l+\Ell_{\inv(\gamma)})\delta_{0,\epsilon}( l_a+\gamma-r)d\gamma
\end{aligned}
\end{equation*}

Now taking the limit $\epsilon\downarrow 0$, and using the continuity of $\gamma\mapsto  \E_a u(l+\Ell_{\inv(\gamma)})$\footnote{The continuity follows from the same analysis as \cite[footnote 1, pg. 9]{bhs21}, which we include for completeness. By assumption, $u$ is compactly supported, so it suffices to show that for a sufficiently large $T$, $\E_{a,l} u(L_{\inv(\gamma)\wedge T})$ is continuous. Since $u$ is Lipschitz, it suffices to show that $\E_{a,l} \|L_{\inv(\gamma-\delta)\wedge T}-L_{\inv(\gamma+\delta)\wedge T}\|_1\rightarrow 0$ as $\delta\downarrow 0$, with $\|\cdot\|_1$ being the 1-norm. Let $J_\delta$ be the event that a jump occurs during the time interval $[\gamma-\delta,\gamma+\delta]$ at $a$. Then, using the total expectation rule with respect to the event of $J_\delta$ and $J_\delta^c$, we have
\[\E_{a,l} \|L_{\inv(\gamma-\delta)\wedge T}-L_{\inv(\gamma+\delta)\wedge T}\|_1\leq \delta\P_{a,l}(J_\delta^c)+T\P_{a,l}(J_\delta)\leq \delta+T\cdot O(\delta)=O(\delta).\]
}, we get
\[\lim_{\epsilon\downarrow 0} \int_0^\infty\E_a u(l+\Ell_{\inv(\gamma)})\delta_{0,\epsilon}( l_a+\gamma-r)d\gamma=\E_a u(l+\Ell_{\inv(r-l_a)}).\]

Therefore, taking the limit $\epsilon\downarrow 0$ in the previous equation, we have
\[\big[\delta_0(l_a) \E_a u(l+\Ell_{\inv(r-l_a)})\big]_\mu
=\big[ \delta_0(l_a-r) \E_a u(l+\Ell_{\inv(r-l_a)})\big]_\mu.\]

Since $\inv(r-l_a)=\inv(0)=0$ when $l_a=r$, and $\inv(r-l_a)=\inv(r)$ when $l_a=0$, we conclude
\[\big[ \delta_0(l_a) \E_a u(l+\Ell_{\inv(r)})\big]_\mu=\big[ \delta_r(l_a)  u(l)\big]_\mu.\]

Now the left-hand side corresponds to expectation with respect to a conditioned 1-permanental vector $\ell$ with kernel $G^h=(-Q^h)^{-1}$, where $Q^h:=Q-he_ae_a^T$, conditioned on $\ell_a=0$. The right-hand side, however, corresponds to expectation with respect to $\ell$ conditioned on $\ell_a=r$. To see this, let us decompose the Laplacian $Q$ as stated in \cref{eqn:decomp}(here we identify node $1$ in the decomposition of \cref{eqn:decomp} with node $a$ mentioned above without loss of generality). Recall that we have Lemma \ref{lem:conditioned}, with $l_*:=(l_i:\,i\neq a)$ here,
\[\rho(l_*|l_a=r)=:\rho_r(l_*) =  \int_{(0,2\pi)^{n-1}}\frac{|-Q_{**}|}{(2\pi)^{n-1}}\,
\exp\Big(\langle (\phi_*+ \sqrt{r}Q_{**}^{-T}Q_{a*}^h), Q_{**}(\bar{\phi}_*+ \sqrt{r}Q_{**}^{-1}Q_{*a}^h)\rangle\Big)d\theta_*.\]

Here we used $Q_{**}^h=Q_{**}$. To apply this formula, note that
\begin{equation*}
\begin{aligned}
\big[ \delta_0(l_a) & \E_a u(l+\Ell_{\inv(r)})\big]_\mu
=\frac{1}{(2\pi)^n}\int_{(0,2\pi)^{n}}\int_{\mathbb{R}_+^{n}} \delta_0(l_a)\E_a u(l+\Ell_{\inv(r)}) e^{\langle \phi,Q\bar{\phi}\rangle}dld\theta\\
&=\frac{1}{(2\pi)^n}\int_{(0,2\pi)^{n-1}}\int_{\mathbb{R}_+^{n-1}} \bigg(\int_{(0,2\pi)}\int_{\mathbb{R}_+} \delta_0(l_a)\E_a u(l+\Ell_{\inv(r)}) e^{\langle \phi,Q\bar{\phi}\rangle}dl_ad\theta_a \bigg)dl_*d\theta_*\\
&=\frac{1}{(2\pi)^{n-1}}\int_{(0,2\pi)^{n-1}}\int_{\mathbb{R}_+^{n-1}} \E_a u(l+\Ell_{\inv(r)}) e^{\langle \phi_*,Q_{**}\bar{\phi}_*\rangle}dl_*d\theta_*.\\
\end{aligned}
\end{equation*}

Here we integrated out $l_a$ and $\theta_a$ using the fact that
\[\int_{\mathbb{R}_+} \delta_0(l_a)\E_a u(l+\Ell_{\inv(r)}) e^{\langle \phi,Q\bar{\phi}\rangle}dl_a= \E_a u(\tilde{l}+\Ell_{\inv(r)}) e^{\langle \phi_*,Q_{**}\bar{\phi}_*\rangle}.\]

Here $\tilde{l}$ is $l$ under the restriction $l_a=0$. We can use Lemma \ref{lem:conditioned} with $r=0$ as
\begin{equation*}
\begin{aligned}
\big[ \delta_0(l_a) \E_a u(l+\Ell_{\inv(r)})\big]_\mu
&=\frac{1}{|-Q_{**}|}\int_{\mathbb{R}_+^{n-1}} \E_a u(\tilde{l}+\Ell_{\inv(r)})\rho_0(l_*)dl_*\\
&=\frac{1}{|-Q_{**}|}\,\E_a\big[ u(\ell+\Ell_{\inv(r)})\,|\,\ell_a=0\big].
\end{aligned}
\end{equation*}

Similarly, we know
\begin{equation*}
\begin{aligned}
\big[ \delta_r(l_a)  u(l)\big]_\mu
&=\frac{1}{(2\pi)^n}\int_{(0,2\pi)^{n}}\int_{\mathbb{R}_+^{n}} \delta_r(l_a)u(l) e^{\langle \phi,Q\bar{\phi}\rangle}dld\theta\\
&=\frac{1}{(2\pi)^n}\int_{(0,2\pi)^{n-1}}\int_{\mathbb{R}_+^{n-1}} \bigg(\int_{(0,2\pi)}\int_{\mathbb{R}_+} \delta_r(l_a)u(l) e^{\langle \phi,Q\bar{\phi}\rangle}dl_ad\theta_a \bigg)dl_*d\theta_*\\
&=\frac{1}{(2\pi)^n}\int_{(0,2\pi)^{n-1}}\int_{\mathbb{R}_+^{n-1}} \bigg(\int_{(0,2\pi)} u(l') e^{\langle \phi',Q\bar{\phi}'\rangle}d\theta_a \bigg)dl_*d\theta_*.
\end{aligned}
\end{equation*}

Here $l'$ is $l$ restricted to $l_a=r$. Similarly for $\phi'$ and $\bar{\phi}'$. We now continue to compute that
\begin{equation*}
\begin{aligned}
& \frac{1}{(2\pi)^n}\int_{(0,2\pi)^n} u(l') e^{\langle \phi',Q\bar{\phi}'\rangle}d\theta_a d\theta_*\\
&=\frac{u(l')e^{Q_{aa}r}}{(2\pi)^n}\int_{(0,2\pi)^n}\exp\bigg(\sum_{i,j\neq a} Q_{ij}\sqrt{l_i l_j}e^{\compi (\theta_i-\theta_j)}
+\sum_{i\neq a}\big(Q_{ia}\sqrt{l_i r}e^{\compi(\theta_i-\theta_a)}+Q_{ai}\sqrt{l_i r}e^{\compi(\theta_a-\theta_i)}\big) \bigg)d\theta_a d\theta_*\\
&=\frac{u(l')e^{Q_{aa}r}}{(2\pi)^n}\int_{(0,2\pi)^n}\exp\bigg(\sum_{i,j\neq a} Q_{ij}\sqrt{l_i l_j}e^{\compi (\theta_i'-\theta_j')}
+\sum_{i\neq a}\big(Q_{ia}\sqrt{l_i r}e^{\compi \theta_i'}+Q_{ai}\sqrt{l_i r}e^{-\compi \theta_i'}\big) \bigg)d\theta_a' d\theta_*'\\
&=\frac{u(l')e^{Q_{aa}r}}{(2\pi)^{n-1}}\int_{(0,2\pi)^{n-1}}\exp\bigg(\sum_{i,j\neq a} Q_{ij}\sqrt{l_i l_j}e^{\compi (\theta_i'-\theta_j')}
+\sum_{i\neq a}\big(Q_{ia}\sqrt{l_i r}e^{\compi \theta_i'}+Q_{ai}\sqrt{l_i r}e^{-\compi \theta_i'}\big) \bigg) d\theta_*'\\
&=\frac{u(l')e^{Q_{aa}r}}{|-Q_{**}|e^{(Q_{aa}^h+(G_{aa}^h)^{-1})r}}\rho_r(l_*).
\end{aligned}
\end{equation*}

In the second step, we used the rotation trick to construct $\theta'$ as in the proof of Lemma \ref{lem:conditioned}. In the third step, we used $Q_{ij}^h=Q_{ij}$ unless $i=j=a$. We also used $G_{aa}^h=|G^h||-Q_{**}^h|=|G^h||-Q_{**}|$ and the following intermediate step in Lemma \ref{lem:conditioned}:
\[\rho_r(l_*)
=\int_{(0,2\pi)^{n-1}}\frac{G_{aa}^h e^{(Q_{aa}^h+(G_{aa}^h)^{-1})r}}{(2\pi)^{n-1}|G^h|}\, \exp\bigg(\sum_{i,j\neq a} Q_{ij}\sqrt{l_i l_j}e^{\compi (\theta_i'-\theta_j')} +\sum_{i\neq a}\big(Q_{ia}\sqrt{l_i r}e^{\compi \theta_i'}+Q_{ai}\sqrt{l_i r}e^{-\compi \theta_i'}\big) \bigg)d\theta_*'.\]

Note that we have $G_{aa}^h=\frac{1}{h}$ since we kill the chain at rate $h$ at state $a$. Therefore, $Q_{aa}-Q_{aa}^h-(G_{aa}^h)^{-1}=h-(G_{aa}^h)^{-1}=0$. This allows us to conclude that
\[\big[ \delta_r(l_a)  u(l)\big]_\mu=\frac{1}{|-Q_{**}|}\int_{\mathbb{R}_+^n}u(l')\rho_r(l_*)dl_*=\frac{1}{|-Q_{**}|}\E\big[u(\ell)\,|\,\ell_a=r\big]\]

This proves the theorem.
\end{proof}

\subsection{non-reversible version of Eisenbaum's isomorphism theorem}

To the best of our knowledge, the non-reversible version of Eisenbaum's isomorphism theorem hasn't been discovered previously, probably because one cannot interpret one side of the equation probabilistically in this case.

\vspace{.3em} \begin{thm}[non-reversible Eisenbaum's isomorphism]
\label{thm:eisenbaum}
Consider an irreducible continuous-time Markov chain $X_t$ on discrete space $[n]$ with Laplacian $Q$. Let $\zeta$ be an independent exponential random variable with mean $h^{-1}$ for some $h>0$. For any $ u:\orthp\rightarrow \mathbb{C}$ with property $\mathfrak{F}$ and any $a\in [n], r>0$, we have
\[
\bigg[ \bigg(\frac{\phi_a+\sqrt{r}}{\sqrt{r}}\bigg) u\big(|\phi+\sqrt{r}\allone|^2\big)\bigg]_{\mu_{h\allone}}
= \big[\E_a^*\big[ u\big(|\phi+\sqrt{r}\allone|^2+\Ell_\zeta\big)\big]\big]_{\mu_{h\allone}}.\]
\end{thm}

\begin{proof}
In Lemma \ref{lem:ward}, we take $g(\phi)=e^{h\sum_i |\phi_i|^2-h\sum_i|\phi_i-\sqrt{r}|^2}=e^{\sqrt{r}h\sum_i (\phi_i+\bar{\phi}_i)-nrh}$, and $f(i,l)= u(l)e^{-h\sum_jl_j}\ind_{i=a}$.\footnote{The function $g$ does not decay exponentially fast as
required in Lemma \ref{lem:ward}, but one can check that the sub-Gaussianity of $f$ will kill any exponential growth in $g$.}  This yields
\[\big[ \phi_a u(\phi\bar{\phi})e^{-h\sum_i|\phi_i-\sqrt{r}|^2}\big]_\mu=\sqrt{r}h\sum_i\bigg[ g(\phi)\int_0^\infty\E_{i,\phi\bar{\phi}}\big[ u(\Ell_t)e^{-h t}\ind_{X_t=a}\big]dt\bigg]_\mu.\]

Now note that $g(\phi)\E_{i,\phi\bar{\phi}}[ u(\Ell_t)e^{-ht}\ind_{X_t=a}]=e^{-h\sum_i|\phi_i-\sqrt{r}|^2}\E_i[ u(\phi\bar{\phi}+\Ell_t)e^{-ht}\ind_{X_t=a}]$. Hence,
\[\big[ \phi_a u(\phi\bar{\phi})e^{-h\sum_i|\phi_i-\sqrt{r}|^2}\big]_\mu=\sqrt{r}h\sum_i\bigg[ e^{-h\sum_i|\phi_i-\sqrt{r}|^2}\int_0^\infty\E_i\big[ u(\phi\bar{\phi}+\Ell_t)e^{-ht}\ind_{X_t=a}\big]dt\bigg]_\mu.\]

Make the change of variable $\phi'=\phi-\sqrt{r}\allone$, and notice that $\langle \phi',Q\bar{\phi}'\rangle=\langle \phi,Q\bar{\phi}\rangle$.
Dividing both sides by $|-Q_{h\allone}|=|hI-Q|$ gives
\[\bigg[ \bigg(\frac{\phi_a'+\sqrt{r}}{\sqrt{r}}\bigg) u\big(|\phi'+\sqrt{r}\allone|^2\big)\bigg]_{\mu_{h\allone}}
= h\sum_i\bigg[\int_0^\infty\E_i\big[ u\big(|\phi'+\sqrt{r}\allone|^2+\Ell_t\big)e^{-ht}\ind_{X_t=a}\big]dt\bigg]_{\mu_{h\allone}}.\]

Reverse the chain on the right-hand side\footnote{Recall that the generator for the reverse process is $Q^T$, and the transition kernel is $P_t^*(i,j)=(e^{tQ^T})_{i,j}$. Since the original Markov process has transition kernel $P_t(j,i)=(e^{tL})_{j,i}$, we know that $P_t^{*}(i,j)=P_t(j,i)$ and they should yield exactly the same bridge measure on paths of length $t$ under path reversal.}, and get

\begin{equation*}
\begin{aligned}
\bigg[ \bigg(\frac{\phi_a'+\sqrt{r}}{\sqrt{r}}\bigg) u\big(|\phi'+\sqrt{r}\allone|^2\big)\bigg]_{\mu_{h\allone}}
&= h\sum_i\bigg[\int_0^\infty\E_a^*\big[ u\big(|\phi'+\sqrt{r}\allone|^2+\Ell_t\big)e^{-ht}\ind_{X_t=i}\big]dt\bigg]_{\mu_{h\allone}}\\
&= h\bigg[\int_0^\infty\E_a^*\big[ u\big(|\phi'+\sqrt{r}\allone|^2+\Ell_t\big)e^{-ht}\big]dt\bigg]_{\mu_{h\allone}}.
\end{aligned}
\end{equation*}

The change of summation and integral can be justified by the bounded convergence theorem, due to the sub-Gaussianity of the density of $[\cdot]_{\mu_{h\allone}}$ and property $\mathfrak{F}$ of the function $u$.

Let $\zeta$ be an exponential r.v. with mean $h^{-1}$, independent of the recurrent random walk. Since we have
\[\E_a^*\big[ u\big(|\phi'+\sqrt{r}\allone|^2+\Ell_\zeta\big)\big]
=h\int_0^\infty\E_a^*\big[ u\big(|\phi'+\sqrt{r}\allone|^2+\Ell_t\big)e^{-ht}\big]dt,\]

we conclude that
\[
\bigg[ \bigg(\frac{\phi_a'+\sqrt{r}}{\sqrt{r}}\bigg) u\big(|\phi'+\sqrt{r}\allone|^2\big)\bigg]_{\mu_{h\allone}}
= \big[\E_a^*\big[ u\big(|\phi'+\sqrt{r}\allone|^2+\Ell_\zeta\big)\big]\big]_{\mu_{h\allone}}.\]
\end{proof}

\vspace{.3em} \begin{rem}
There is a generalized version of \cref{thm:eisenbaum} above. Consider any $h_1,\cdots, h_n\geq 0$ such that $\sum_{i=1}^n h_i>0$. Take $g(\phi)=e^{\sum_i h_i|\phi_i|^2-\sum_ih_i|\phi_i-\sqrt{r}|^2}=e^{\sqrt{r}\sum_i h_i(\phi_i+\bar{\phi}_i)-r\sum_i h_i}$, and $f(i,l)= u(l)e^{-\sum_jh_jl_j}\ind_{i=a}$ in Lemma \ref{lem:ward}. This yields
\[\big[ \phi_a u(\phi\bar{\phi})e^{-\sum_ih_i|\phi_i-\sqrt{r}|^2}\big]_\mu=\sqrt{r}\sum_ih_i\bigg[ g(\phi)\int_0^\infty\E_{i,\phi\bar{\phi}}\big[ u(\Ell_t)e^{-\sum_j h_j\Ell_t^j}\ind_{X_t=a}\big]dt\bigg]_\mu.\]

Note that $g(\phi)\E_{i,\phi\bar{\phi}}[ u(\Ell_t)e^{-\sum_j h_j\Ell_t^j}\ind_{X_t=a}]=e^{-\sum_ih_i|\phi_i-\sqrt{r}|^2}\E_i[ u(\phi\bar{\phi}+\Ell_t)e^{-\sum_j h_j\Ell_t^j}\ind_{X_t=a}]$. Hence,
\[\big[ \phi_a u(\phi\bar{\phi})e^{-\sum_ih_i|\phi_i-\sqrt{r}|^2}\big]_\mu=\sqrt{r}\sum_ih_i\bigg[ e^{-\sum_ih_i|\phi_i-\sqrt{r}|^2}\int_0^\infty\E_i\big[ u(\phi\bar{\phi}+\Ell_t)e^{-\sum_j h_j\Ell_t^j}\ind_{X_t=a}\big]dt\bigg]_\mu.\]

Making the change of variable $\phi'=\phi-\sqrt{r}\allone$, and
dividing both sides by $|-Q_{\bm{h}}|$ gives
\[\bigg[ \bigg(\frac{\phi_a'+\sqrt{r}}{\sqrt{r}}\bigg) u\big(|\phi'+\sqrt{r}\allone|^2\big)\bigg]_{\mu_{\bm{h}}}
= \sum_ih_i\bigg[\int_0^\infty\E_i\big[ u\big(|\phi'+\sqrt{r}\allone|^2+\Ell_t\big)e^{-\sum_j h_j\Ell_t^j}\ind_{X_t=a}\big]dt\bigg]_{\mu_{\bm{h}}}.\]
\end{rem}

\section{Comparison inequalities for permanental processes}

Kahane-type inequalities are inequalities between expectations of functions of centered multivariate Gaussian random variables when one assumes that the covariance matrices of the Gaussians can be compared in some sense, see e.g.~\cite[Lecture 5]{biskup17}. Slepian's lemma~\cite{biskup17, slepian62} is a famous special case. In \cite{eisenbaum13}, a Kahane-type inequality and a version of Slepian's lemma for 4-permanental processes has been given. However, 4-permanental processes are not directly related to the isomorphism theorems, which limits the applications of these results. We use the density-formula approach to derive analogous results for 1-permanental processes. We remark that Eisenbaum's method used a lemma of \cite{sato99}, which is not satisfied by 1-permanental vectors.

For any function $f:\mathbb{R}_+^n\rightarrow \mathbb{R}$, we say it has subexponential growth for $k\leq 2$ orders of partial derivatives if (i.)$f$ has subexponential growth; (ii.)both $\{\partial_i f :=\frac{\partial}{\partial l_i}f, i\in [n]\}$ and $\{\partial_{ij} f :=\frac{\partial^2}{\partial l_i\partial l_j}f, i,j\in [n]\}$ exist and have subexponential growth. Recall that any infinitely divisible kernel $G$ of a $1$-permanental random vector has an equivalent kernel $\tilde{G}$ that is an inverse M-matrix with a positive definite symmetric part according to \cref{rem-infdiv}. We have the following comparison inequality.

\vspace{.3em} \begin{lem}[Kahane-type inequality for 1-permanental vectors]
\label{lem:kahane}
Consider two kernels for $1$-permanental vectors $\{G_i, i=0,1\}$ and an interpolation $\{G_\alpha, \alpha\in[0,1]\}$ between them such that $G_\alpha$ is an inverse M-matrix with a positive definite symmetric part for any $\alpha\in[0,1]$, and $\alpha\mapsto G_{\alpha,ij}$ is of class $C^1([0,1])$, with $\frac{d}{d\alpha}G_{\alpha, ij}$ having a constant sign for any $i,j\in[n]$. Let $\{\ell_\alpha,\, \alpha\in[0,1]\}$ be the corresponding 1-permanental vectors.
Let $f:\mathbb{R}_+^n\rightarrow \mathbb{R}$ be any function with subexponential growth for $k\leq 2$ orders of partial derivatives, such that for any $i\in[n]$, either of the following is true:
\begin{enumerate}
    \item $\frac{d}{d\alpha}G_{\alpha,ii}\geq 0$, $\forall \alpha\in[0,1]$, and $\partial_i f(l)+l_i\,\partial_{ii}f(l)\geq 0$, $\forall l\in\mathbb{R}_+^n$;
    \item $\frac{d}{d\alpha}G_{\alpha,ii}\leq 0$, $\forall \alpha\in[0,1]$, and $\partial_i f(l)+l_i\,\partial_{ii}f(l)\leq 0$, $\forall l\in\mathbb{R}_+^n$;
\end{enumerate}

and for any $i\neq j$, either of the following is true:
\begin{enumerate}
    \item $\frac{d}{d\alpha}G_{\alpha,ij}\geq 0$, $\forall \alpha\in[0,1]$, and $\partial_{ij}f(l)\geq 0$, $\forall l\in\mathbb{R}_+^n$;
    \item $\frac{d}{d\alpha}G_{\alpha,ij}\leq 0$, $\forall \alpha\in[0,1]$, and $\partial_{ij}f(l)\leq 0$, $\forall l\in\mathbb{R}_+^n$.
\end{enumerate}

Then we have $\E[f(\ell_1)]\geq \E[f(\ell_0)]$.
\end{lem}

\begin{proof}
Let
\begin{equation}        \label{eq:qalpha}
Q_{\alpha}:=-(G_{\alpha})^{-1}.
\end{equation}
By \cref{thm:lejan}, we know that
\[\E[f(\ell_\alpha)]=\frac{1}{(2\pi\compi)^n|G_\alpha|}\int_{\mathbb{C}^n}f(\phi\bar{\phi})e^{\langle \phi,Q_{\alpha}\bar{\phi}\rangle}d\bar{\phi} d\phi .\]

Therefore, we have
\[\E[f(\ell_1)]-\E[f(\ell_0)]
= \int_0^1\Big(\frac{\partial}{\partial \alpha} \E[f(\ell_\alpha)]\Big)d\alpha\\
= \int_0^1\frac{\partial}{\partial \alpha}\bigg(\int_{\mathbb{C}^n}\frac{1}{(2\pi\compi)^n|G_\alpha|}f(\phi\bar{\phi})e^{\langle \phi,Q_{\alpha}\bar{\phi}\rangle}d\bar{\phi} d\phi \bigg)d\alpha.\]

Computing the derivative as\footnote{Recall that for a 
 }
\[\frac{\partial}{\partial \alpha}\bigg(\frac{1}{|G_\alpha|}e^{\langle \phi,Q_{\alpha}\bar{\phi}\rangle}\bigg)=\tr\bigg(\frac{d}{d G_{\alpha}}\bigg(\frac{1}{|G_\alpha|}e^{-\langle \phi,G_\alpha^{-1}\bar{\phi}\rangle}\bigg)\, \frac{d G_\alpha}{d\alpha}\bigg),\]

we know that
\[
\E[f(\ell_1)]-\E[f(\ell_0)]
= \frac{1}{(2\pi\compi)^n}
\int_0^1\int_{\mathbb{C}^n}f(\phi\bar{\phi})\, \tr\bigg(\frac{d}{d G_{\alpha}}\bigg(\frac{1}{|G_\alpha|}e^{-\langle \phi,G_\alpha^{-1}\bar{\phi}\rangle}\bigg)\, \frac{d G_\alpha}{d\alpha}\bigg)d\bar{\phi} d\phi d\alpha.
\]

We remark that interchanging
the derivative and the integral can be justified by the absolute integrability of the partial 
The detailed justification of this fact can be found in the appendix (\cref{apx:kahane}).

The Jacobi formula (see, e.g., Section 0.8.2 of \cite{hj12}) states that for invertible matrix $A$, $\frac{\partial}{\partial A_{ij}} |A|=(adj(A))_{ji}$, where the adjugate is $adj(A):=|A|A^{-1}$. It is then easy to verify that
\[\frac{d}{d G_{\alpha}}\bigg(\frac{1}{|G_\alpha|}\bigg)
=-\frac{1}{|G_\alpha|^2}\, \frac{d|G_\alpha|}{d G_\alpha} =-\frac{1}{|G_\alpha|^2}\, |G_\alpha|(G_\alpha^{-1}) =\frac{1}{|G_\alpha|}Q_{\alpha}.\]

Using the fact that $\frac{d}{d A}\tr(A^{-1}B)=-A^{-1}BA^{-1}$, we know that
\[
\frac{d}{d G_{\alpha}} e^{-\langle \phi,G_\alpha^{-1}\bar{\phi}\rangle}
=-e^{-\langle \phi,G_\alpha^{-1}\bar{\phi}\rangle}\, \frac{d}{dG_\alpha}\tr\big(G_\alpha^{-1}\bar{\phi}\phi^T\big) = e^{-\langle \phi,G_\alpha^{-1}\bar{\phi}\rangle}\, G_\alpha^{-1}\bar{\phi}\phi^TG_\alpha^{-1} = e^{-\langle \phi,G_\alpha^{-1}\bar{\phi}\rangle}\, Q_{\alpha}\bar{\phi}\phi^TQ_{\alpha}.
\]

Let
\begin{equation}        \label{eq:dmualpha}
d\mu_\alpha(\phi):=(2\pi\compi)^{-n}|G_\alpha|^{-1}e^{-\langle \phi,G_\alpha^{-1}\bar{\phi}\rangle}d\bar{\phi}d\phi
\end{equation}
denote the normalized twisted Gaussian density. We know that
\begin{equation*}
\begin{aligned}
\E[f(\ell_1)]& -\E[f(\ell_0)]
=\int_0^1\int_{\mathbb{C}^n}f(\phi\bar{\phi})\,
\tr\big( (Q_{\alpha} + Q_{\alpha}\bar{\phi}\phi^TQ_{\alpha}) \frac{d G_\alpha}{d\alpha}\big) d\mu_\alpha(\phi)d\alpha.\\
=&\int_0^1\int_{\mathbb{C}^n}f(\phi\bar{\phi})\,
\sum_{i,j}\frac{d G_{\alpha,ij}}{d\alpha}\bigg( Q_{\alpha,ji} + \sum_{i',j'}Q_{\alpha,jj'}Q_{\alpha,i'i}\phi_{i'}\bar{\phi}_{j'} \bigg)\,
d\mu_\alpha(\phi)d\alpha \\
=&\int_0^1\int_{\mathbb{C}^n}\,
\sum_{i,j}\frac{d G_{\alpha,ij}}{d\alpha}\bigg( Q_{\alpha,ji}f(\phi\bar{\phi}) + \sum_{i'}Q_{\alpha,i'i}\Big(-\delta_{i'j}f(\phi\bar{\phi})-\phi_{i'}\bar{\phi}_j\partial_j f(\phi\bar{\phi})\Big) \bigg)
d\mu_\alpha(\phi)d\alpha.
\end{aligned}
\end{equation*}

In the third step, we used the following integration by parts for any $i,j\in[n]$:
\begin{equation*}
\begin{aligned}
\int_{\mathbb{C}^n} & \frac{\partial}{\partial \phi_j}\big(\phi_{i'}f(\phi\bar{\phi})e^{-\langle \phi, G_\alpha^{-1}\bar{\phi}\rangle}\big)d\bar{\phi} d\phi =0\quad\Rightarrow\\
0=& \int_{\mathbb{C}^n} \delta_{i'j}f(\phi\bar{\phi})d\mu_\alpha(\phi) +
\int_{\mathbb{C}^n}\phi_{i'}\bar{\phi}_j\partial_j f(\phi\bar{\phi})d\mu_\alpha(\phi) +
\int_{\mathbb{C}^n}\phi_{i'}f(\phi\bar{\phi})\bigg(\sum_{j'}Q_{\alpha, jj'}\bar{\phi}_{j'}\bigg)d\mu_\alpha(\phi).
\end{aligned}
\end{equation*}

We can simplify by canceling terms and doing another integration-by-parts, and get
\begin{equation*}
\begin{aligned}
\E[f(\ell_1)]& -\E[f(\ell_0)]
=\int_0^1\int_{\mathbb{C}^n}\,
\sum_{i,j}\frac{d G_{\alpha,ij}}{d\alpha}\bigg(  -\sum_{i'}Q_{\alpha,i'i}\phi_{i'}\bar{\phi}_j\partial_j f(\phi\bar{\phi}) \bigg)
d\mu_\alpha(\phi)d\alpha\\
=&\int_0^1\int_{\mathbb{C}^n}\,
\sum_{i,j}\frac{d G_{\alpha,ij}}{d\alpha}\bigg(  \delta_{ij}\partial_j f(\phi\bar{\phi}) + \phi_i\bar{\phi}_j\partial_{ij} f(\phi\bar{\phi}) \bigg)
d\mu_\alpha(\phi)d\alpha.
\end{aligned}
\end{equation*}

The integration-by-parts formula used here is that for any $i,j\in[n]$,
\begin{equation*}
\begin{aligned}
\int_{\mathbb{C}^n} & \frac{\partial}{\partial \bar{\phi}_i}\big(\bar{\phi}_j \partial_j f(\phi\bar{\phi})e^{-\langle \phi, G_\alpha^{-1}\bar{\phi}\rangle}\big)d\bar{\phi} d\phi =0\quad\Rightarrow\\
0=& \int_{\mathbb{C}^n} \delta_{ij}\partial_j f(\phi\bar{\phi})d\mu_\alpha(\phi) +
\int_{\mathbb{C}^n}\phi_i\bar{\phi}_j\partial_{ij} f(\phi\bar{\phi})d\mu_\alpha(\phi) +
\int_{\mathbb{C}^n}\bar{\phi}_j\partial_j f(\phi\bar{\phi})\bigg(\sum_{i'}Q_{\alpha, i'i}\phi_{i'}\bigg)d\mu_\alpha(\phi).
\end{aligned}
\end{equation*}

We remark that the validity of these integration-by-parts procedures is guaranteed by the subexponential growth conditions. Given that $G_\alpha$ is an infinitely divisible kernel of a $1$-permanental vector with a positive definite symmetric part, and $\partial_{ij}f(l)$ is subexponential for any $i,j$, using Lemma \ref{lem:nonneg}, we know that for any $i\neq j$,
\[ \int_{\mathbb{C}^n} \frac{d G_{\alpha,ij}}{d\alpha}\phi_i\bar{\phi}_j\partial_{ij} f(\phi\bar{\phi})
d\mu_\alpha(\phi) = \frac{d G_{\alpha,ij}}{d\alpha}\int_{\mathbb{C}^n}\phi_i\bar{\phi}_j\partial_{ij} f(\phi\bar{\phi}) d\mu_\alpha(\phi)\geq 0.\]

On the other hand, for $i\in [n]$, by \cref{thm:lejan}, we have
\[\int_{\mathbb{C}^n}\, \frac{d G_{\alpha, ii}}{d \alpha}\Big(\partial_i f(\phi\bar{\phi})+\phi_i\bar{\phi}_i\partial_{ii} f(\phi\bar{\phi}) \Big)
d\mu_\alpha(\phi)=\E \bigg[\frac{d G_{\alpha, ii}}{d \alpha}\big(\partial_i f(\ell_\alpha)+\ell_i\,\partial_{ii}f(\ell_\alpha)\big)\bigg]\geq 0.\]

It is then straightforward to check that the conditions on $f$ stated in the lemma imply that $\E[f(\ell_\alpha)]$ is increasing with $\alpha$. Therefore, $\E[f(\ell_1)]\geq \E[f(\ell_0)]$ and the claim is proved.
\end{proof}

\vspace{.3em} \begin{rem}[Sign patterns I]
Given any sign pattern $\{s_{ij}\in\{-1,1\},\, i,j\in V\}$, we know there indeed
exist nontrivial functions that
satisfy the equations $s_{ii}(\partial_i f(l)+l_i\,\partial_{ii}f(l))> 0$, and $s_{ij}\partial_{ij}f(l)>0$, $\forall l\in \mathbb{R}_+^n$, for any $i\neq j\in V$. Consider the function
\[f(l)=\sum_{(i,j):i\neq j}-s_{ij}\log(1+l_i+l_j)+\sum_i s_{ii}(2ne^{l_i}).\]

Then we know that $\forall i\in V$,
\begin{equation*}
\begin{aligned}
s_{ii}(\partial_i f(l)+l_i\,\partial_{ii}f(l))
&= \sum_{j\neq i}-\frac{s_{ii}s_{ij}}{(1+l_i+l_j)}+2ne^{l_i}+\sum_{j\neq i}\frac{s_{ii}s_{ij}l_i}{(1+l_i+l_j)^2}+2nl_ie^{l_i}\\
&\geq -(n-1)+2n -(n-1)>0.
\end{aligned}
\end{equation*}

In the first step, we used $s_{ii}^2=1$. In the second step, we used the fact that $\frac{1}{(1+l_i+l_j)}\in(0,1]$, $\frac{l_i}{(1+l_i+l_j)^2}\in(0,1]$, $e^{l_i}\geq 1$, and $l_ie^{l_i}\geq 0$. Also, note that $\forall i\neq j\in V$,
\[s_{ij}\partial_{ij}f(l)=\frac{1}{(1+l_i+l_j)^2}>0, \forall l\in\mathbb{R}_+^n.\]
\end{rem}

\vspace{.3em} \begin{rem}
Let us consider the following instance for Lemma \ref{lem:kahane}. Let $P_0, P_1\in M_n(\mathbb{R}_+)$ be two sub-stochastic matrices(i.e., every row sum $\leq 1$) such that $P_{0,ij}\leq P_{1,ij}\forall i,j\in[n]$. For any $s>1$, consider $G_0:=(sI-P_0)^{-1}$ and $G_1:=(sI-P_1)^{-1}$. They are both inverse M-matrices according to the definition of an M-matrix \cite{plemmons77}, since $G_0^{-1}=sI-P_0$ where $\rho(P_0)\leq \|P_0\|_\infty:=\max_i\sum_j |P_{0,ij}|\leq 1<s$ and the same for $G_1^{-1}$. Moreover, we know that $\{P_\alpha:=\alpha P_1+(1-\alpha)P_0,\,\alpha\in[0,1]\}$ is a family of sub-stochastic matrices, and $\{G_\alpha:=(sI-P_\alpha)^{-1},\,\alpha\in[0,1]\}$ is a family of inverse M-matrices. Using the Neumann series expansion, we have
    \[\frac{d}{d\alpha}G_\alpha=\frac{d}{d\alpha}\sum_{i=0}^\infty s^{-(i+1)}P_\alpha^i=\sum_{i=0}^\infty s^{-(i+1)}\frac{d}{d\alpha}P_\alpha^i=\sum_{i=1}^\infty s^{-(i+1)}\sum_{j=1}^{i-1}P_\alpha^{j}\Big(\frac{d}{d\alpha}P_\alpha\Big)P_\alpha^{i-1-j}.\]
Since $\frac{d}{d\alpha}P_\alpha=P_1-P_0$ is entry-wise nonnegative, we know that $\frac{d}{d\alpha}G_{\alpha,ij}\geq 0,\forall i,j\in[n]$. One can then plug in any function of interest that satisfies for any $i\in[n],\,\partial_i f(l)+l_i\,\partial_{ii}f(l)\geq 0, \forall l\in\mathbb{R}_+^n$ and for any $i\neq j\in[n],\,\partial_{ij} f(l)\geq 0, \forall l\in\mathbb{R}_+^n$ to obtain non-trivial comparison inequalities.
\end{rem}

\vspace{.3em} \begin{rem}[Sign patterns II]
Consider the following family of inverse M-matrices in $M_n(\mathbb{R}_+)$, where $n\geq 3$. Given $0<x<1$, let $y:=\sqrt{\frac{1}{n-2}\,x+\frac{(n-3)}{(n-2)}\,x^2}$. Consider the family of matrices $\mathcal{B}:=\{B: B_{ii}=1,\forall i\in [n]; 0<x<B_{ij}<y<1, \forall i\neq j\}$. This family of matrices is proved to be inverses of strictly diagonally dominant (both by rows and columns) M-matrices in \cite[Theorem 2]{will77}. We claim that any M-matrix $A$ that is strictly diagonally dominant both by rows and columns has a positive definite symmetric part, since for any $x\in\mathbb{R}^n$ such that $x\neq 0$,

\begin{align*}
\langle x,Ax\rangle
&= \sum_{i}|A_{ii}|x_i^2-\sum_{i,j:i<j}(|A_{ij}|+|A_{ji}|)x_ix_j \\
&> \sum_{i}\sum_{j\neq i}\frac{1}{2}(|A_{ij}|+|A_{ji}|)x_i^2-\sum_{i,j:i<j}(|A_{ij}|+|A_{ji}|)x_ix_j\\
&= \sum_{i,j:i<j}\frac{1}{2}(|A_{ij}|+|A_{ji}|)(x_i^2+x_j^2)-\sum_{i,j:i<j}(|A_{ij}|+|A_{ji}|)x_ix_j\\
&= \sum_{i,j:i<j}\frac{1}{2}(|A_{ij}|+|A_{ji}|)(x_i^2-2x_ix_j+x_j^2)\\
&= \sum_{i,j:i<j}\frac{1}{2}(|A_{ij}|+|A_{ji}|)(x_i-x_j)^2\geq 0.\\
\end{align*}

In the first step, we used the sign pattern of $A$ as an M-matrix. In the second step, we used strict diagonal dominance in rows and columns, together with the fact that $x\neq 0$. Therefore, we know that any matrix in the family $\mathcal{B}$, as an inverse of a strictly diagonally dominant (both by rows and columns) M-matrix, should have a positive definite symmetric part.

Moreover, for any two matrices $G_0, G_1\in\mathcal{B}$, we can construct a linear path between them as $G_\alpha:=\alpha G_1+(1-\alpha)G_0, \forall \alpha\in(0,1)$. The whole path will be within family $\mathcal{B}$ since this family forms a convex set. Moreover, the sign of $\frac{d}{d\alpha}G_\alpha=G_1-G_0$ will be constant entry-wise along the path. We can then vary $G_0$ and $G_1$ to obtain
whatever arbitrary off-diagonal sign pattern we want.
\end{rem}

Note that Eisenbaum derived the comparison lemma (Lemma 3.1 in \cite{eisenbaum13}) for 4-permanental vectors by directly working with the Fourier transform, which is very different from our approach. Also, their lemma requires the function $f$ to be bounded. With the density formula, we substantially relax this condition to functions with subexponential growth here.

Convolutional tricks can be used to generalize the above lemma to $k$-permanental vectors for $k\in\mathbb{N}_+$ under the condition
\[G_{1,ii}>(<)G_{0,ii} \Rightarrow k\,\partial_i f+l_i\,\partial_{ii}f\geq(\leq)0; \qquad G_{1,ij}>(<)G_{0,ij} \Rightarrow \partial_{ij}f\geq(\leq)0.\]

To see this, let us consider the case $k=2$. The general case follows the same idea. Let $\ell_i$ and $\ell_i'$ be independent 1-permanental vectors with kernel $G_i$, for $i=0,1$. Then $\ell_i+\ell_i'$ will be a 2-permanental vector with kernel $G_i$, which can be seen by calculating the moment generating function. Using the \cref{thm:lejan} for both $\ell_1,\ell_1'$ and $\ell_0,\ell_0'$, we have
\begin{equation*}
\begin{aligned}
\E\big[f( \ell_1 +\ell_1')\big]-& \E\big[f(\ell_0+\ell_0')\big]\\
=&\int_0^1\int_{\mathbb{C}^{2n}}\,
\frac{1}{(2\pi \compi)^{2n}}f(\phi\bar{\phi}+\phi'\bar{\phi}') \frac{\partial}{\partial \alpha}\bigg(\, \frac{1}{|G_\alpha|^2}e^{\langle \phi,Q_{\alpha}\bar{\phi} \rangle+\langle \phi',Q_{\alpha}\bar{\phi}' \rangle}\bigg)
d\bar{\phi} d\phi d\phi'd\bar{\phi}'d\alpha\\
=&\int_0^1\int_{\mathbb{C}^n}\int_{\mathbb{C}^n}\,
\frac{1}{(2\pi \compi)^n}f(\phi\bar{\phi}+\phi'\bar{\phi}') \frac{\partial}{\partial \alpha}\bigg(\, \frac{1}{|G_\alpha|}e^{\langle \phi,Q_{\alpha}\bar{\phi} \rangle} \bigg)d\mu_\alpha(\phi')
d\bar{\phi} d\phi d\alpha \,+ \\
&\hspace{1em}\int_0^1\int_{\mathbb{C}^n}\int_{\mathbb{C}^n}\,
\frac{1}{(2\pi \compi)^n}f(\phi\bar{\phi}+\phi'\bar{\phi}') \frac{\partial}{\partial \alpha}\bigg(\, \frac{1}{|G_\alpha|}e^{\langle \phi',Q_{\alpha}\bar{\phi}' \rangle} \bigg)d\mu_\alpha(\phi)
d\phi' d\bar{\phi}'d\alpha.
\end{aligned}
\end{equation*}

Now note that the proof of Lemma \ref{lem:kahane} shows that for any function $g:\mathbb{R}_+^n\rightarrow\mathbb{R}$ with subexponential growth for $k\leq 2$ orders of partial derivative, and any $\alpha\in(0,1)$, we have
\[\int_{\mathbb{C}^n}\frac{1}{(2\pi\compi)^n}\frac{\partial}{\partial \alpha}\Big(\frac{1}{|G_\alpha|}e^{\langle\phi, Q_\alpha\bar{\phi}\rangle}\Big)g(\phi\bar{\phi}) d\bar{\phi}d\phi = \int_{\mathbb{C}^n}\,
\sum_{i,j}\frac{d G_{\alpha,ij}}{d\alpha}\bigg(  \delta_{ij}\partial_j g(\phi\bar{\phi}) + \phi_i\bar{\phi}_j\partial_{ij} g(\phi\bar{\phi}) \bigg)
d\mu_\alpha(\phi).\]

Specifically, if for fixed $l=\phi\bar{\phi}$, we let $g(l):=\int_{\mathbb{C}^n}f(\phi\bar{\phi}+\phi'\bar{\phi'})d\mu_\alpha(\phi')$ and apply the above equation to $g(\phi\bar{\phi})$, we get
\begin{equation*}
\begin{aligned}
& \int_0^1\int_{\mathbb{C}^n}\int_{\mathbb{C}^n}\,
\frac{1}{(2\pi \compi)^n}f(\phi\bar{\phi}+\phi'\bar{\phi}') \frac{\partial}{\partial \alpha}\bigg(\, \frac{1}{|G_\alpha|}e^{\langle \phi,Q_{\alpha}\bar{\phi} \rangle} \bigg)d\mu_\alpha(\phi')
d\bar{\phi} d\phi d\alpha \\
=&\int_0^1\int_{\mathbb{C}^{2n}}\,
\sum_{i,j}\frac{d G_{\alpha,ij}}{d\alpha}\bigg(  \delta_{ij}\partial_j f(\phi\bar{\phi}+\phi'\bar{\phi}') + \phi_i\bar{\phi}_j\partial_{ij} f(\phi\bar{\phi}+\phi'\bar{\phi}') \bigg)
d\mu_\alpha(\phi)d\mu_\alpha(\phi')d\alpha.
\end{aligned}
\end{equation*}

Similarly, we can fix $\phi,\bar{\phi}$ and let $g(\phi'\bar{\phi}'):=\int_{\mathbb{C}^n}f(\phi\bar{\phi}+\phi'\bar{\phi'})d\mu_\alpha(\phi)$. This gives us
\begin{equation*}
\begin{aligned}
& \int_0^1\int_{\mathbb{C}^n}\int_{\mathbb{C}^n}\,
\frac{1}{(2\pi \compi)^n}f(\phi\bar{\phi}+\phi'\bar{\phi}') \frac{\partial}{\partial \alpha}\bigg(\, \frac{1}{|G_\alpha|}e^{\langle \phi',Q_{\alpha}\bar{\phi}'\rangle} \bigg)d\mu_\alpha(\phi)
d\bar{\phi}' d\phi' d\alpha \\
=&\int_0^1\int_{\mathbb{C}^{2n}}\,
\sum_{i,j}\frac{d G_{\alpha,ij}}{d\alpha}\bigg(  \delta_{ij}\partial_j f(\phi\bar{\phi}+\phi'\bar{\phi}') + \phi_i'\bar{\phi}_j'\partial_{ij} f(\phi\bar{\phi}+\phi'\bar{\phi}') \bigg)
d\mu_\alpha(\phi)d\mu_\alpha(\phi')d\alpha.
\end{aligned}
\end{equation*}

Therefore, we conclude that
\begin{equation*}
\begin{aligned}
\E\big[f(& \ell_1 +\ell_1')\big]-\E\big[f(\ell_0+\ell_0')\big]\\
=&\int_0^1\int_{\mathbb{C}^{2n}}\,
\sum_{i,j}\frac{d G_{\alpha,ij}}{d\alpha}\bigg(  2\delta_{ij}\partial_j f(\phi\bar{\phi}+\phi'\bar{\phi}') + (\phi_i\bar{\phi}_j+\phi_i'\bar{\phi}_j')\partial_{ij} f(\phi\bar{\phi}+\phi'\bar{\phi}') \bigg)
d\mu_\alpha(\phi)d\mu_\alpha(\phi')d\alpha.
\end{aligned}
\end{equation*}

Note that $\forall i\in [n]$,
\[\Delta G_{ii}(2\partial_i f(l)+l_i\partial_{ii} f(l))\geq 0,\]

by assumption. Also, for any $i\neq j$,
\[
\int_{\mathbb{C}^{2n}}\, \phi_i\bar{\phi}_j \frac{d G_{\alpha,ij}}{d\alpha} \partial_{ij}f(\phi\bar{\phi}  +\phi'\bar{\phi}')
d\mu_\alpha(\phi)d\mu_\alpha(\phi') =\int_{\mathbb{C}^{2n}}\, \phi_i\bar{\phi}_j\frac{d G_{\alpha,ij}}{d\alpha} \E\big[ \partial_{ij}f(\phi\bar{\phi}+\ell_\alpha')\big]d\mu_\alpha(\phi).
\]

Here $\ell_\alpha'$ is a 1-permanental process with kernel $G_\alpha$. Using Lemma \ref{lem:nonneg}, we conclude that the integral above has a fixed sign. This guarantees
\[\int_{\mathbb{C}^{2n}}\, \frac{d G_{\alpha,ij}}{d\alpha}(\phi_i\bar{\phi}_j+\phi_i'\bar{\phi}_j')\partial_{ij} f(\phi\bar{\phi}+\phi'\bar{\phi}')
d\mu_\alpha(\phi)d\mu_\alpha(\phi')\geq 0.\]

Hence $\E[f(\ell_\alpha+\ell_\alpha')]$ is increasing with $\alpha\in(0,1)$, and $\E[f(\ell_1+\ell_1')]\geq \E[f(\ell_0+\ell_0')]$. This verifies the case $k=2$. The general case where $k$ is any positive integer follows from almost the same arguments as above.

\subsection{Slepian's lemma for 1-permanental processes}

We now turn to prove an analog of Slepian's lemma for 1-permanental processes. To do so, we first prove a weaker version as below.

\vspace{.3em} \begin{lem}[Weak comparision of 1-permanental vectors]
\label{lem:weak_comp}
In the same setting as Lemma \ref{lem:kahane}, if, further, we have $G_{0,ii}=G_{1,ii}$ for all $i\in[n]$ and $G_{1,ij}\geq G_{0,ij}$ for all $i\neq j$, then
\[\sup_i \ell_{0,i}\succeq \sup_i\ell_{1,i}.\]
\end{lem}

\begin{proof}
Given $\epsilon>0$, consider the function $g_\epsilon:\mathbb{R}\rightarrow [0,1]$ as follows,
\[g_\epsilon(x):=\begin{cases}
1& \text{ if }x< 0,\\
1-6\epsilon^{-5}x^5+15\epsilon^{-4}x^4-10\epsilon^{-3}x^3& \text{ if }x\in[0,\epsilon),\\
0& \text{ if }x\geq \epsilon.
\end{cases}\]

We know that $g_\epsilon$ is decreasing and twice differentiable. Given $s_1, \cdots, s_n\geq 0$, define $f_\epsilon(l_1,..., l_n) :=\prod_{i=1}^n g_\epsilon(l_i-s_i)$. Note (even though we do not use this property) that for any $i\in[n]$, we have $\partial_i f(l)\leq 0$. And for any $i\neq j$, $\partial_{ij}f(l)\geq 0$. Hence, by Lemma \ref{lem:kahane}, we have
\[\E[f_\epsilon(\ell_1)]\geq \E[f_\epsilon(\ell_0)].\]

Taking the limit as $\epsilon\downarrow 0$ on both sides, we have
\[\lim_{\epsilon\downarrow 0}\E[f_\epsilon(\ell_1)]\geq \lim_{\epsilon\downarrow 0}\E[f_\epsilon(\ell_0)].\]

By the dominated convergence theorem, since $\E[|f_\epsilon(\ell_1)|]\leq \E[1]=1$, we can
interchange the integral with the limit and get
\[\E[\lim_{\epsilon\downarrow 0}f_\epsilon(\ell_1)]\geq \E[\lim_{\epsilon\downarrow 0}f_\epsilon(\ell_0)].\]

But $\lim_{\epsilon\downarrow 0}f_\epsilon(l)=\prod_{i=1}^n\ind_{(-\infty, s_i]}(l_i)$. We conclude that
\[\P(\ell_{1,i}\leq s_i, \forall i\in[n])\geq \P(\ell_{0,i}\leq s_i, \forall i\in[n]).\]

Since this is true for any $\{s_i\geq 0 : i\in[n]\}$, we conclude that $\sup_i \ell_{0,i}\succeq \sup_i\ell_{1,i}$.
\end{proof}

Boosting the above lemma using a scaling trick, we have the following Slepian lemma for 1-permanental vectors.

\vspace{.3em} \begin{thm}[Slepian lemma for 1-permanental vectors]
\label{thm:slepian}
In the same setting as Lemma \ref{lem:kahane}, if, further, we have $G_{1,ii}\leq G_{0,ii}$ for all $i\in[n]$ and $G_{1,ij} \geq G_{0,ij}$ for all $i\neq j$, then
\[\sup_i \ell_{0,i}\succeq \sup_i\ell_{1,i}.\]
\end{thm}

\begin{proof}
To see this, we scale the 1-permanental vector $\ell_0$ by $c=(c_1,\cdots, c_n)$ where $c_i:=\frac{G_{1,ii}}{G_{0,ii}}\leq 1$(we use the convention that $\frac{0}{0}=1$). We claim that the new vector $c\odot\ell_0$, where $\odot$ is the standard Hadamard product between vectors, is also a 1-permanental vector, since for any $\lambda\in\mathbb{R}_+^n$, we have
\[\E[e^{-\langle \lambda, c\odot \ell_0\rangle}]=\E[e^{-\langle \lambda\odot c, \ell_0\rangle}]=|I+\Lambda C G_0|^{-1}.\]

Here, we defined $C:=\diag(c)$ and used the moment generating function of $\ell_0$. Then by definition, $c\odot\ell_0$ is a 1-permanental vector with kernel $\tilde{G}:=CG_0$. Note that $\forall i\in[n]$, $\tilde{G}_{ii}=c_iG_{0,ii}=G_{1,ii}$; and $\forall i\neq j$, $\tilde{G}_{ij}=c_iG_{0,ij}\leq G_{0,ij}\leq G_{1,ij}$. Therefore, we can use Lemma \ref{lem:weak_comp} to conclude that $\sup_i c_i\ell_{0,i}\succeq \sup_i \ell_{1,i}$. Here for two scalar random variables $X,Y$, $X$ (first-order) stochastically dominates $Y$, i.e., $X\succeq Y$ iff $\forall x\in \mathbb{R}$, $\P(X\geq x)\geq \P(Y\geq x)$. Let $c_*:=\sup_i c_i=\sup_i \frac{G_{1,ii}}{G_{0,ii}}$, then we conclude that
\[\sup_i \ell_{0,i}\succeq c_*\sup_i \ell_{0,i}\succeq\sup_i c_i\ell_{0,i}\succeq \sup_i \ell_{1,i}.\]
\end{proof}

\section{Symmetrization bound on cover time}

Our main result in this section will be a symmetrization bound on the cover time of a non-reversible Markov chain on a finite set. An important property of this upper bound is that it coincides with that for a random walk on a graph, i.e. a reversible Markov chain on a finite set, given in \cite{dlp11}, which is tight up to constant factors in the reversible case. To arrive at our result we apply the density formula, with the generalized second Ray-Knight isomorphism theorem for non-reversible Markov chains, to study the local time process at inverse local time.

Let $P=(P_{ij}:\, i,j\in[n])$ be an irreducible transition probability matrix on the state space $[n]$, with stationary distribution $\pi=(\pi_i:\, i\in[n])$. Recall that the normalized Laplacian is $Q=\Pi(P-I)$, where $\Pi$ denotes the diagonal matrix with diagonal entries given by the $\pi_i$. Let $\{ X_t, t \ge 0\}$ denote the continuous time Markov chain with rate matrix $Q$. Note that the cover time being discussed is defined in terms of the underlying discrete-time Markov chain with transition probability matrix $P$, as is conventional in the literature on cover times.

We will decompose $Q$ as in \cref{eqn:decomp}. The Green's kernel associated with killing upon hitting node 1 is
\begin{equation}    \label{eq:definetildeG}
\tilde{G}=(-Q_{**})^{-1}.
\end{equation}
Recall that the local time process is defined as in
\cref{eqn:localtime}.
For a fixed state $a \in [n]$, the inverse local time for any $r\geq 0$ is $\inv(r):=\inf\{t : \Ell_t^a>r\}$. We may as well renumber the states so that $a =1$. We will then consider the permanental process with kernel $G$ defined by
\[G:=\begin{pmatrix}
1 & \allone^T\\ \allone & \tilde{G}+\allone\allone^T
\end{pmatrix}.\]

One can think of it as the Green's kernel associated with killing once the process $\{X_t: t\geq 0\}$ has spent an exponentially distributed random time with mean $1$
at node 1 \cite{ek09}.
Equivalently, $G=(e_1e_1^T-Q)^{-1}$ in light of \cref{thm:rayknight}. Let $\ell$ be a 1-permanental process with kernel $G$. Then from \cref{thm:rayknight} we know
\[\{\Ell_{\inv(r)}+\ell\,|\,\ell_1=0\}\,\stackrel{d}{=}\,\{\ell\,|\,\ell_1=r\}.\]

Note that $\{\ell\,|\,\ell_1=0\}$ is distributed as a permanental vector with kernel $\tilde{G}$, so we have a density formula for it. For $\{\ell\,|\,\ell_1=r\}$, we will use Lemma \ref{lem:conditioned} to bound its density.

We first derive a tail bound for $\inv (t)$, which will be useful when bounding the cover time.

\vspace{.3em} \begin{lem}[Tail bound lemma]
\label{lem:tail}
For an irreducible Markov chain on state space $[n]$, let $S_{**}=\frac{1}{2}(Q_{**}+Q_{**}^T)$ and denote by $\sigma^2$ the largest diagonal term of $\frac{1}{2}(-S_{**})^{-1}$. Then for any $\lambda>0$, we have
\[\P\bigg(\bigg(\sum_{i=1}^n \pi_i\Ell_{\inv(r)}^i\bigg)^{\frac{1}{2}}-\sqrt{r}\geq \sqrt{\lambda}\sigma\bigg)\leq 5\gamma e^{-\lambda/8}.\]
Here
\begin{equation}    \label{eq:definegamma}
\gamma:=\frac{|2\tilde{G}|}{|\tilde{G}+\tilde{G}^T|}.
\end{equation}
\end{lem}

\begin{proof}
First, in light of the proof of \cref{thm:rayknight}, we have

\[\rho_r(l_2,\cdots, l_n) =  \int_{(0,2\pi)^{n-1}}\frac{1}{(2\pi)^{n-1}|\tilde{G}|}\,
\exp\Big(\langle \phi_*- \sqrt{r}\allone, Q_{**}(\bar{\phi}_*- \sqrt{r}\allone)\rangle\Big)d\theta_*.\]

Also, we have
\[\bigg\{\sum_{i=1}^n \pi_i\Ell_{\inv(r)}^i+\sum_{i=1}^n\pi_i\ell_i\,\bigg|\,\ell_1=0\bigg\}
\stackrel{d.}{=}
\bigg\{\sum_{i=1}^n\pi_i\ell_i\,\bigg|\,\ell_1=r\bigg\}.\]

Hence we conclude that $\{\sum_{i=1}^n\pi_i\ell_i\,|\,\ell_1=r\}\succeq\sum_{i=1}^n \pi_i\Ell_{\inv(r)}^i$. Let $\mathcal{R}:=\{l:\, \sum_{i=1}^n \pi_il_i\geq c\}$. Then we have the following bound:
\begin{equation*}
\begin{aligned}
\P\bigg(\sum_{i=1}^n \pi_i\Ell_{\inv(r)}^i\geq c\bigg) & \leq \int_\mathcal{R} \rho_r(l_2,\cdots,l_n)dl_*\\
& = \int_\mathcal{R} \int_{(0,2\pi)^{n-1}}\frac{1}{(2\pi)^{n-1}|\tilde{G}|}\,
\exp\Big(\langle \phi_*- \sqrt{r}\allone, Q_{**}(\bar{\phi}_*- \sqrt{r}\allone)\rangle\Big)d\theta_* dl_*\\
& \leq \int_\mathcal{R} \int_{(0,2\pi)^{n-1}}\frac{1}{(2\pi)^{n-1}|\tilde{G}|}\,
\exp\Big(\langle \phi_*- \sqrt{r}\allone, S_{**}(\bar{\phi}_*- \sqrt{r}\allone)\rangle\Big)d\theta_* dl_*\\
& = \gamma \int_\mathcal{R} \int_{(0,2\pi)^{n-1}}\frac{|-S_{**}|}{(2\pi)^{n-1}}\,
\exp\Big(\langle \phi_*- \sqrt{r}\allone, S_{**}(\bar{\phi}_*- \sqrt{r}\allone)\rangle\Big)d\theta_* dl_*.
\end{aligned}
\end{equation*}

Note that the integrand is now the density of $Z+{{\sqrt{r}\allone}\choose{0}}$, where $Z=\eta+\compi \xi$ is a complex Gaussian random vector with $\eta,\xi$ being i.i.d. $\mathcal{N}(0,\frac{1}{2}(-S_{**})^{-1})$ random variables. Therefore,
\begin{equation*}
\begin{aligned}
\P\bigg(\sum_{i=1}^n \pi_i\Ell_{\inv(r)}^i\geq c\bigg)
& \leq \gamma\P\big((\eta+\sqrt{r}\allone)^2+\xi^2)\in \mathcal{R}\big)\\
& = \gamma \P\big((\eta^2+\xi^2+2\sqrt{r}\eta+r\allone)\in \mathcal{R}\big)\\
& = \gamma \P\bigg(\sum_{i=1}^n \pi_i(\eta_i^2+\xi^2+2\sqrt{r}\eta_i+r)\geq c\bigg).
\end{aligned}
\end{equation*}

From Claim 2.2 of \cite{ding14}, we know for any $\lambda>0$, $\P(\sum_{i=1}^n \pi_i\eta_i^2\geq \lambda \sigma^2)\leq 2 e^{-\lambda/4}$, where $\sigma^2=\max_i\E[\eta_i^2]$. The same is true for $\sum_{i=1}^n \pi_i\xi_i^2$. We also know that $\sum_{i=1}^n \pi_i\eta_i\prec\mathcal{N}(0,\sigma^2)$, hence $\P(\sum_i\pi_i\eta_i\geq \sqrt{\lambda}\sigma)\leq e^{-\lambda/4}$. Therefore,
\begin{equation*}
\begin{aligned}
\P\bigg(\sum_i&\pi_i\eta_i^2 +\sum_i\pi_i\xi^2+2\sqrt{r}\sum_i\pi_i\eta_i+r\geq \lambda  \sigma^2+2\sqrt{\lambda r} \sigma+r\bigg)\\
&\leq \P\bigg(\sum_i\pi_i\eta_i^2\geq \frac{1}{2}\lambda   \sigma^2\bigg)+ \P\bigg(\sum_i\pi_i\xi^2\geq \frac{1}{2}\lambda  \sigma^2\bigg)+
\P\bigg(\sum_i \pi_i\eta_i\geq \sqrt{\lambda}\sigma \bigg)\\
&\leq 4e^{-\lambda/8}+e^{-\lambda/4}\leq 5e^{-\lambda/8}.
\end{aligned}
\end{equation*}

Hence $\P\big(\sum_{i=1}^n \pi_i\Ell_{\inv(r)}^i\geq (\sqrt{r}+\sqrt{\lambda }\sigma)^2\big)\leq 5\gamma e^{-\lambda/8}$.
\end{proof}

Let us denote the random cover time as $\tau_\cov:=\inf\{\sum_i\pi_i\Ell_s^i:\, s\geq 0,\,\Ell_s^i>0,\forall\,i\in[n]\}$\footnote{Since we are working with a variable jump rate Markov chain in this section, we need $\pi_i$-weighted summation to recover the cover time for a discrete Markov chain.}. The cover time $t_\cov$ is just the maximum expectation of $\tau_\cov$ maximized over all starting states $X_0\in[n]$. We denote the random cover-and-return time as
\[\tau_\cov^+:=\inf\bigg\{\sum_i\pi_i\Ell_s^i:\, s\geq 0,\, X_s=X_0,\,\Ell_s^i>0,\forall\,i\in[n]\bigg\}.\]

Similarly, we can define the cover-and-return time $t_\cov^+$. We remark that using the Markov property, one can easily see $t_\cov\leq \E(\tau_{cov}^+)\leq 2t_\cov$ for any irreducible finite-state chains, with any starting state. See \cite{dlp11} for a short proof.\footnote{They only considered symmetric chains, but their proof holds for any non-symmetric chain as well.} Given $j\in[n]$, the random hitting time, $\tau_{\textnormal{hit}}(j)$, is defined as
\[\tau_{\textnormal{hit}}(j):=\inf\bigg\{\sum_i\pi_i\Ell_s^i:\, s\geq 0,\, X_s=j\bigg\}.\]

Similarly, define the hitting time as $t_{\textnormal{hit}}=\sup_{i,j\in[n]}\E[\tau_{\textnormal{hit}}(j)|X_0=i]$. Let us also denote by $t_\cov'$ and $t_{\textnormal{hit}}'$ the cover time and the hitting time of the symmetrized Markov chain with kernel $\frac{Q+Q^T}{2}$, respectively.

We now prove the following symmetrization upper bound on the cover time.

\vspace{.3em} \begin{thm}[Symmetrization upper bound on cover time]
\label{thm:sym-cover}
For an irreducible Markov chain on state space $[n]$ with generator $Q$, we decompose $Q$ as in \cref{eqn:decomp}. Let $S$ be defined in terms of $Q$ as in equation \eqref{eq:defineS}, and decompose $S$ similarly to $Q$. We
write $M:=\E\sup_i\eta_i$ for $\eta\sim\mathcal{N}(0,\frac{1}{2}(-S_{**})^{-1}))$, and let $\sigma^2$ be the maximal diagonal term in $\frac{1}{2}(-S_{**})^{-1})$. Then we have the following upper bound on the cover time
\[t_\cov=O(M^2+\sigma^2\log\gamma)=O(t_\cov'\log\gamma),\]
where $\gamma$ is defined in equation \eqref{eq:definegamma}.
\end{thm}

\begin{proof}
To give an upper bound on the cover time, we need to
\begin{itemize}
    \item bound $\{\sup_i\ell_i\,|\,\ell_1=0\}$ from above asymptotically;
    \item bound $\{\inf_i\ell_i\,|\,\ell_1=t\}$ from below asymptotically.
\end{itemize}

Let $\eta, \xi$ denote independent $\mathcal{N}(0,\frac{1}{2}(-S_{**})^{-1})$ random vectors. Then by standard results about Gaussian suprema (see, for example, Theorem 5.8 of \cite{blm13}), we know that for all $\lambda > 0$ we have
\[\P\Big(\sup_i \eta_i>\E\sup_i\eta_i+\sqrt{\lambda}\sigma\Big)\leq e^{-\lambda/2}.\]

Since $X$ is symmetric, $-\inf_i \eta_i\stackrel{d.}{=}\sup_i \eta_i$. Hence we also have
\[\P\Big(-\inf_i \eta_i>\E\sup_i\eta_i+\sqrt{\lambda}\sigma\Big)\leq e^{-\lambda/2}.\]

By an application of the union bound, we deduce
\[\P\Big(\sup_i |\eta_i|>M+\sqrt{\lambda}\sigma\Big)=\P\Big(\sup_i \eta_i^2>\big(M+\sqrt{\lambda}\sigma\big)^2\Big)\leq 2e^{-\lambda/2}.\]

For any $c>0$, let $\mathcal{R}':=\{l:\, \sup_i l_i\geq c\}$. Then we have the following bound, where $\tilde{G}$ is defined as in equation \eqref{eq:definetildeG},
\begin{equation*}
\begin{aligned}
\P\bigg(\sup_i\ell_i \geq c\,\Big|\,\ell_1=0\Big) & = \int_{\mathcal{R}'} \rho_0(l_2,\cdots,l_n)dl_*\\
& = \int_{\mathcal{R}'} \int_{(0,2\pi)^{n-1}}\frac{1}{(2\pi)^{n-1}|\tilde{G}|}\,
\exp\big(\langle \phi_*, Q_{**}\bar{\phi}_*\rangle\big)d\theta_* dl_*\\
& \leq \gamma \int_{\mathcal{R}'} \int_{(0,2\pi)^{n-1}}\frac{|-S_{**}|}{(2\pi)^{n-1}}\,
\exp\big(\langle \phi_*, S_{**}\bar{\phi}_*\rangle\big)d\theta_* dl_*\\
& = \gamma\P\big((\eta^2+\xi^2)\in \mathcal{R}'\big).
\end{aligned}
\end{equation*}

Recall that here $\eta, \xi$ are independent $\mathcal{N}(0,\frac{1}{2}(-S_{**})^{-1})$ random vectors. This allows us to conclude that
\begin{equation*}
\begin{aligned}
\P\Big(\sup_i\ell_i\geq 2(M+& \sqrt{\lambda}\sigma)^2\,\Big|\,\ell_1=0\Big)\\
&\leq \gamma\P\Big(\sup_i(\eta_i^2+\xi_i^2)\geq 2(M+\sqrt{\lambda}\sigma)^2\Big)\\
&\leq \gamma\P\Big(\sup_i\eta_i^2+\sup_i\xi_i^2\geq 2(M+\sqrt{\lambda}\sigma)^2\Big)\\
&\leq \gamma\P\Big(\sup_i\eta_i^2\geq (M+\sqrt{\lambda}\sigma)^2\Big)+\gamma\P\Big(\sup_i\xi_i^2\geq (M+\sqrt{\lambda}\sigma)^2\Big)\\
&\leq 4\gamma e^{-\lambda/2}.
\end{aligned}
\end{equation*}

Similarly, note that for any $t>0$,
\begin{equation*}
\begin{aligned}
\P\Big(\inf_i\ell_i\leq -2\sqrt{t}M-& 2\sqrt{\lambda t}\sigma+t\,\Big|\,\ell_1=t\Big)\\
&\leq \gamma\P\Big(\inf_i\big((\eta_i+\sqrt{t})^2+\xi_i^2\big)\leq -2\sqrt{t}M-2\sqrt{\lambda t}\sigma+t\Big)\\
&\leq \gamma\P\Big(\inf_i(2\sqrt{t}\eta_i+t)\leq -2\sqrt{t}M-2\sqrt{\lambda t}\sigma+t\Big)\\
&= \gamma\P\Big(\inf_i\eta_i\leq -M-\sqrt{\lambda}\sigma\Big)\\
&\leq \gamma e^{-\lambda/2}.
\end{aligned}
\end{equation*}

Now we take $t=9(M+\sqrt{\lambda}\sigma)^2$, so that $-2\sqrt{t}M- 2\sqrt{\lambda t}\sigma+t=3(M+\sqrt{\lambda}\sigma)^2> 2(M+\sqrt{\lambda}\sigma)^2$. Therefore, with probability $\geq 1-5\gamma e^{-\lambda/2}$, we have $\tau_\cov^+ \leq \sum_i\pi_i\Ell_{\inv(t)}^i$ for such $t$. Also, by Lemma \ref{lem:tail}, with probability $\geq 1-5\gamma e^{-\lambda/8}$, we know that $\sum_i\pi_i\Ell_{\inv(t)}^i\leq (\sqrt{t}+\sqrt{\lambda}\sigma)^2=(3M+4\sqrt{\lambda}\sigma)^2$.

So in conclusion,

\[\P(\tau_\cov^+\geq (3M+4\sqrt{\lambda}\sigma)^2)\leq 10\gamma e^{-\lambda/8},\quad\forall \lambda>0.\]

Specifically, we have $\P(\tau_\cov^+\geq 2((3M)^2+\lambda (4\sigma)^2))\leq \P(\tau_\cov^+\geq (3M+4\sqrt{\lambda}\sigma)^2)\leq 10\gamma e^{-\lambda/8}$.

Since $\gamma\geq 1$, we know that when $\lambda \geq 16\log \gamma$, we have $\gamma e^{-\lambda/8}\leq (\gamma e^{-\lambda/16}) e^{-\lambda/16}\leq  e^{-\lambda/16}$. Therefore, we conclude that
\begin{equation*}
\begin{aligned}
\E\tau_\cov^+ & = \int_0^\infty \P(\tau_\cov^+\geq t)dt \\
& = \int_0^{18M^2} \P(\tau_\cov^+\geq t)dt + \int_0^\infty \P(\tau_\cov^+\geq 18M^2+32\lambda \sigma^2)\cdot 32\sigma^2 d\lambda\\
& \leq \int_0^{18M^2} 1\cdot dt + \int_0^{16\log\gamma} 1\cdot 32\sigma^2 d\lambda + \int_{16\log\gamma}^\infty \P(\tau_\cov^+\geq 18M^2+32\lambda \sigma^2)\cdot 32\sigma^2 d\lambda\\
& \leq 18M^2 + 16\log\gamma \cdot 32\sigma^2 + \int_{16\log\gamma}^\infty 10e^{-\lambda/16}\cdot 32\sigma^2 d\lambda\\
& \leq 18M^2 + 512\sigma^2\log\gamma + \int_0^\infty 10e^{-\lambda/16}\cdot 32\sigma^2 d\lambda\\
& = 18M^2 + 512\sigma^2\log\gamma+5120\sigma^2.
\end{aligned}
\end{equation*}

Therefore, we deduce that $\E\tau_\cov^+=O(M^2+\sigma^2\log\gamma)$. Since $t_\cov'=\Theta(M^2+\sigma^2)$, we obtain $t_\cov=\Theta(\E\tau_\cov^+)=O(t_\cov'\log\gamma)$.
\end{proof}

We remark that when $\gamma= 2^{O(M^2/\sigma^2)}$, i.e., when the chain does not deviate too much from symmetry, using the remarkable theorem of Ding-Lee-Peres\cite{dlp11}, we have $t_\cov=O(t_\cov')$. We also note that since non-reversibility reduces hitting time, i.e., $t_{\textnormal{hit}}\leq t_{\textnormal{hit}}'$(Corollary 3.1 of \cite{hm17}), then by Mathews' bound (see, e.g., Theorem 11.2 of \cite{lp17}), we know
\[t_\cov=O(t_\textnormal{hit}\log n)=O(t_\textnormal{hit}'\log n)=O(t_\cov'\log n).\]

However, this upper bound does not provide a smooth transition between the tight reversible bound and the non-reversible scenarios as \cref{thm:sym-cover}.

\section{Summary}

In this paper, we used a density formula approach to study several problems on non-reversible finite-state Markov chains and the associated 1-permanental processes, as well as a larger class of related 1-permanental processes.
After deriving a density formula for this broad family of 1-permanental processes, we used it, together with non-reversible Ward identities, to give unified proofs of several isomorphism theorems for non-reversible Markov chains. We then applied the density formulas to derive Kahane-type and
Slepian-type comparison inequalities for 1-permanental processes. Finally, we
applied the technique to bounding the cover time of non-reversible Markov chains. The derived bound becomes tight in the reversible case \cite{dlp11}.

The authors believe that the twisted Gaussian density is a good substitute for the Gaussian free field (GFF) in the non-reversible case (it indeed coincides with GFF when the Markov chain is reversible). In light of this, for many previous results about reversible Markov chains proved via the GFF, one should be able to give a natural generalization (perhaps suboptimal)
in non-reversible scenarios using the twisted Gaussian density instead.

Another intriguing research question is as follows. In the Markovian case, the coordinates of the 1-permanental vector corresponding to the kernel with psd symmetric part are the total loop soup local times where loops are sampled at rate 1. Is there a similar picture for the general kernel with
a symmetric psd part?

\appendix
\section{Absolute integrability for Lemma \ref{lem:ward}.}
\label{apx:sec1}

We adopt the setting and notation of Lemma \ref{lem:ward}. Write \(\phi=u+\compi v\) with \(u,v\in\mathbb{R}^n\), and set
\[
\ell(u,v):=(u_1^2+v_1^2,\dots,u_n^2+v_n^2)\in\mathbb{R}_+^n,
\qquad
r(u,v):=\sum_{j=1}^n (u_j^2+v_j^2).
\]
Then \(\sum_{j=1}^n \phi_j\bar{\phi}_j=r(u,v)\) and \((2\pi\compi)^{-n}\,d\bar{\phi}\,d\phi=\pi^{-n}\,du\,dv\). Also, recalling that \(S:=\frac12(Q+Q^\top)\), see \eqref{eq:defineS}, and using that \(S\) is symmetric, we have
\[
\bigl|e^{\langle \phi,Q\bar{\phi}\rangle}\bigr|
= e^{\langle \phi,S\bar{\phi}\rangle} = e^{\langle u,Su\rangle+\langle v,Sv\rangle}.
\]

We first bound \(T_i g\). Property \(\mathfrak{G}\) implies that, for each \(i\in[n]\), the function \(T_i g\) has polynomial growth. Since \([n]\) is finite, there exists \(m>0\) such that for each \(i\in[n]\) there are constants \(a_i,b_i>0\) with \(|T_i g(\phi)|\le a_i\|\phi\|^{2m}\) whenever \(\|\phi\|>b_i\). Taking maxima over \(i\), and then enlarging the resulting constant to cover the compact set \(\{\phi\in\mathbb{C}^n:\|\phi\|\le \max_i b_i\}\), we obtain a constant \(C_g>0\) such that
\[
|T_i g(\phi)|\le C_g(1+\|\phi\|^{2m}),
\qquad \phi\in\mathbb{C}^n,\ i\in[n].
\]
Since \(\|\phi\|^2=r(u,v)\), this becomes
\[
|T_i g(u+\compi v)|\le C_g(1+r(u,v))^m,
\qquad u,v\in\mathbb{R}^n,\ i\in[n].
\]

Next we bound \(f\). Since \(f\) has property \(\mathfrak{F}\), applying the definition with exponent \(m+2\) shows that for each \(i\in[n]\) there are constants \(a_i',b_i'>0\) such that
\[
|f(i,\ell)|\le a_i'\|\ell\|^{-(m+2)}
\qquad\text{whenever }\|\ell\|>b_i'.
\]
Using again that \([n]\) is finite and that all norms on \(\mathbb{R}^n\) are equivalent, we find constants \(A,B>0\) such that
\[
|f(i,\ell)|\le A\Bigl(\sum_{j=1}^n \ell_j\Bigr)^{-(m+2)}
\qquad\text{whenever }\sum_{j=1}^n \ell_j>B,\ i\in[n].
\]
Enlarging \(A\) if necessary to cover the compact set \(\{\ell\in\mathbb{R}_+^n:\sum_j\ell_j\le B\}\), we obtain \(C_f>0\) such that
\[
|f(i,\ell)|\le C_f\Bigl(1+\sum_{j=1}^n \ell_j\Bigr)^{-(m+2)},
\qquad \ell\in\mathbb{R}_+^n,\ i\in[n].
\]

Now let
\[
I:=\int_0^\infty \sum_i \int_{\mathbb{C}^n}
\bigl|(T_i g)(\phi)\,f_t(i,\phi\bar{\phi})\bigr|\,
(2\pi\compi)^{-n}\bigl|e^{\langle\phi,Q\bar{\phi}\rangle}\bigr|
\,d\bar{\phi}\,d\phi\,dt.
\]
Passing to \((u,v)\)-coordinates, applying \(|\E Z|\le \E|Z|\), and using the two bounds above, we get
\[
\begin{aligned}
I
&=
\int_0^\infty \sum_i \int_{\mathbb{R}^{2n}}
\bigl|(T_i g)(u+\compi v)\bigr|\,
\bigl|\E_{i,\ell(u,v)}(f(X_t,\Ell_t))\bigr|\,
\pi^{-n}e^{\langle u,Su\rangle+\langle v,Sv\rangle}\,du\,dv\,dt \\
&\le
\int_0^\infty \sum_i \int_{\mathbb{R}^{2n}}
\bigl|(T_i g)(u+\compi v)\bigr|\,
\E_{i,\ell(u,v)}|f(X_t,\Ell_t)|\,
\pi^{-n}e^{\langle u,Su\rangle+\langle v,Sv\rangle}\,du\,dv\,dt \\
&\le
C_gC_f
\int_0^\infty \sum_i \int_{\mathbb{R}^{2n}}
(1+r(u,v))^m\,
\E_{i,\ell(u,v)}
\Bigl(1+\sum_{j=1}^n \Ell_t^j\Bigr)^{-(m+2)}
\pi^{-n}e^{\langle u,Su\rangle+\langle v,Sv\rangle}\,du\,dv\,dt.
\end{aligned}
\]
Since the total local time increases at unit rate, we have \(\sum_{j=1}^n \Ell_t^j=\sum_{j=1}^n \ell_j(u,v)+t=r(u,v)+t\). Hence
\[
I
\le
C_gC_f
\int_0^\infty \sum_i \int_{\mathbb{R}^{2n}}
(1+r(u,v))^m(1+r(u,v)+t)^{-(m+2)}
\pi^{-n}e^{\langle u,Su\rangle+\langle v,Sv\rangle}\,du\,dv\,dt.
\]
The integrand is nonnegative, so Tonelli's theorem applies. Therefore
\[
\begin{aligned}
I
&\le
C_gC_f \sum_i \int_{\mathbb{R}^{2n}} (1+r(u,v))^m
\left(\int_0^\infty (1+r(u,v)+t)^{-(m+2)}\,dt\right)
\pi^{-n}e^{\langle u,Su\rangle+\langle v,Sv\rangle}\,du\,dv \\
&=
\frac{C_gC_f}{m+1}
\sum_i \int_{\mathbb{R}^{2n}}
(1+r(u,v))^{-1}
\pi^{-n}e^{\langle u,Su\rangle+\langle v,Sv\rangle}\,du\,dv \\
&\le
\frac{C_gC_f}{m+1}
\sum_i \int_{\mathbb{R}^{2n}}
\pi^{-n}e^{\langle u,Su\rangle+\langle v,Sv\rangle}\,du\,dv .
\end{aligned}
\]
The last integral is finite by the same Gaussian integrability used in Lemma \ref{lem:ward}. Hence \(I<\infty\).

Thus the integrand is absolutely integrable, and Fubini's theorem allows us to exchange the relevant integrals and the summation.

\section{Interchange of Differentiation and Integration in Lemma \ref{lem:kahane}.}
\label{apx:kahane}

Recall that
\[
Q_\alpha:=-(G_\alpha)^{-1},
\qquad
d\mu_\alpha(\phi):=\frac{1}{(2\pi \compi)^n|G_\alpha|} e^{\langle \phi,Q_\alpha\bar\phi\rangle}\,d\bar\phi\,d\phi,
\]
see equations \eqref{eq:qalpha} and \eqref{eq:dmualpha}. For every $\alpha\in[0,1]$, let
\[
F(\alpha):=\E[f(\ell_\alpha)] =\int_{\mathbb C^n} f(\phi\bar\phi)\,d\mu_\alpha(\phi) =\int_{\mathbb C^n} H(\alpha,\phi)\,d\bar\phi\,d\phi,
\]
where
\[
H(\alpha,\phi):=\frac{1}{(2\pi \compi)^n|G_\alpha|} f(\phi\bar\phi)e^{\langle \phi,Q_\alpha\bar\phi\rangle}.
\]

Since $\alpha\mapsto G_\alpha$ is $C^1([0,1])$ entrywise, it is a $C^1$ map into $\mathbb R^{n\times n}$. The determinant map $A\mapsto \det(A)$ is polynomial in the entries of $A$, hence smooth, and the inversion map $A\mapsto A^{-1}$ is smooth on the open set of invertible matrices. Since each $G_\alpha$ is invertible, it follows that
\[
\alpha\mapsto |G_\alpha|^{-1},
\qquad
\alpha\mapsto G_\alpha^{-1},
\qquad
\alpha\mapsto Q_\alpha
\]
are all $C^1([0,1])$. Hence, for each fixed $\phi\in\mathbb C^n$, the map $\alpha\mapsto H(\alpha,\phi)$ is $C^1([0,1])$.

To obtain a uniformly integrable majorant, write
\[
A_\alpha:=-Q_\alpha=G_\alpha^{-1},
\qquad
S_\alpha:=\frac{A_\alpha+A_\alpha^T}{2}.
\]
By Lemma~\ref{lem:inv_psd}, $S_\alpha$ is positive definite for every $\alpha\in[0,1]$. Since $\alpha\mapsto A_\alpha$ is continuous, so is $\alpha\mapsto S_\alpha$. Since the smallest eigenvalue is continuous on the space of real symmetric matrices, the map
\[
\alpha\mapsto \lambda_{\min}(S_\alpha)
\]
is continuous on $[0,1]$. As each $S_\alpha$ is positive definite, $\lambda_{\min}(S_\alpha)>0$ for all $\alpha$, and compactness of $[0,1]$ yields
\[
c:=\min_{\alpha\in[0,1]}\lambda_{\min}(S_\alpha)>0.
\]

Now write
$\phi=u+ \compi v$ with $u,v\in\mathbb R^n$. Since $A_\alpha$ is real,
\[
\Re\langle \phi,A_\alpha\bar\phi\rangle
= u^T A_\alpha u+v^T A_\alpha v = u^T S_\alpha u+v^T S_\alpha v
\ge c(\|u\|^2+\|v\|^2)
= c\|\phi\|^2.
\]
Equivalently,
\[
\Re\langle \phi,Q_\alpha\bar\phi\rangle
\le -c\|\phi\|^2.
\]
Therefore,
\[
\big|e^{\langle \phi,Q_\alpha\bar\phi\rangle}\big|
= e^{\Re\langle \phi,Q_\alpha\bar\phi\rangle}
\le e^{-c\|\phi\|^2},
\qquad \forall \alpha\in[0,1],\ \phi\in\mathbb C^n.
\]

Next, because $\alpha\mapsto G_\alpha$, $\alpha\mapsto Q_\alpha$, and $\alpha\mapsto \frac{d}{d\alpha}G_\alpha$ are continuous on the compact interval $[0,1]$, there exist finite constants
\[
M_0:=\sup_{\alpha\in[0,1]} |G_\alpha|^{-1},\qquad M_1:=\sup_{\alpha\in[0,1]} \|Q_\alpha\|,\qquad M_2:=\sup_{\alpha\in[0,1]} \Big\|\frac{d}{d\alpha}G_\alpha\Big\|.
\]

Using Jacobi's formula (see, e.g., Section 0.8.2 of \cite{hj12}) and
\[
\frac{d}{d\alpha}Q_\alpha
= Q_\alpha\Big(\frac{d}{d\alpha}G_\alpha\Big)Q_\alpha,
\]
we get
\[
\frac{d}{d\alpha}|G_\alpha|^{-1}
= |G_\alpha|^{-1}\tr\Big(Q_\alpha\frac{d}{d\alpha}G_\alpha\Big),
\]
and
\[
\partial_\alpha e^{\langle \phi,Q_\alpha\bar\phi\rangle}
= e^{\langle \phi,Q_\alpha\bar\phi\rangle}
\Big\langle \phi,
Q_\alpha\Big(\frac{d}{d\alpha}G_\alpha\Big)Q_\alpha\bar\phi
\Big\rangle.
\]
Hence
\[
\partial_\alpha H(\alpha,\phi)
=
\frac{1}{(2\pi \compi)^n|G_\alpha|}
f(\phi\bar\phi)e^{\langle \phi,Q_\alpha\bar\phi\rangle}
\Bigg(
\tr\Big(Q_\alpha\frac{d}{d\alpha}G_\alpha\Big)
+
\Big\langle \phi,
Q_\alpha\Big(\frac{d}{d\alpha}G_\alpha\Big)Q_\alpha\bar\phi
\Big\rangle
\Bigg).
\]

Taking absolute values and using
\[
\Big|\tr\Big(Q_\alpha\frac{d}{d\alpha}G_\alpha\Big)\Big|
\le n\|Q_\alpha\|\,\Big\|\frac{d}{d\alpha}G_\alpha\Big\|,
\]
and
\[
\Big|
\Big\langle \phi,
Q_\alpha\Big(\frac{d}{d\alpha}G_\alpha\Big)Q_\alpha\bar\phi
\Big\rangle
\Big|
\le
\|Q_\alpha\|^2\Big\|\frac{d}{d\alpha}G_\alpha\Big\|\,\|\phi\|^2,
\]
we obtain
\[
|\partial_\alpha H(\alpha,\phi)|
\le
C\,|f(\phi\bar\phi)|(1+\|\phi\|^2)e^{-c\|\phi\|^2},
\qquad \forall \alpha\in[0,1],\ \phi\in\mathbb C^n,
\]
for some constant $C<\infty$ independent of $\alpha$ and $\phi$.

By the definition of subexponential growth,
\[
\lim_{\|l\|_2\to\infty}\frac{\log(1+|f(l)|)}{\|l\|_2}=0.
\]
Fix $\varepsilon>0$. Then there exists $R_\varepsilon<\infty$ such that whenever $\|l\|_2\ge R_\varepsilon$,
\[
\frac{\log(1+|f(l)|)}{\|l\|_2}\le \varepsilon,
\]
and hence
\[
|f(l)|\le e^{\varepsilon \|l\|_2},
\qquad \|l\|_2\ge R_\varepsilon.
\]
Since $f$ is continuous, it is bounded on the compact set $\{l\in\mathbb R_+^n:\|l\|_2\le R_\varepsilon\}$; let
\[
M_\varepsilon:=\sup_{\|l\|_2\le R_\varepsilon}|f(l)|<\infty.
\]
Therefore, for all $l\in\mathbb R_+^n$,
\[
|f(l)|\le C_\varepsilon e^{\varepsilon \|l\|_2},
\qquad
C_\varepsilon:=\max\{1,M_\varepsilon\}.
\]
Applying this with $l=\phi\bar\phi=(|\phi_1|^2,\dots,|\phi_n|^2)$, and using
\[
\|\phi\bar\phi\|_2
=
\Big(\sum_{j=1}^n |\phi_j|^4\Big)^{1/2}
\le
\sum_{j=1}^n |\phi_j|^2
=
\|\phi\|_2^2,
\]
we obtain
\[
|f(\phi\bar\phi)|\le C_\varepsilon e^{\varepsilon \|\phi\|_2^2}.
\]
Choose $\varepsilon>0$ so small that $\varepsilon<c/2$. Then
\[
|\partial_\alpha H(\alpha,\phi)|
\le
C\cdot C_\epsilon(1+\|\phi\|_2^2)e^{-(c-\varepsilon)\|\phi\|_2^2},
\qquad \forall \alpha\in[0,1],\ \phi\in\mathbb C^n.
\]
Since $c-\varepsilon>0$ and $\mathbb C^n\simeq\mathbb R^{2n}$, the right-hand side belongs to $L^1(\mathbb C^n)$, because any polynomial times a non-degenerate Gaussian density is integrable.

By the Leibniz integral rule for differentiation under the integral sign (see, e.g., Theorem 2.27 in \cite{fo99}), justified here by the dominated convergence theorem, the exchange of differentiation and integration in the proof of Lemma \ref{lem:kahane} is valid.

\section{Nonnegativity of a certain integral}
\label{apx:sec2}

Consider $G$, which is an inverse M-matrix with a positive definite symmetric part. Let $A:=G^{-1}$. Then $A$ also has a positive definite symmetric part, as proved in Lemma \ref{lem:inv_psd} in \cref{app:misc}. It follows that the complex measure $d\mu(\phi):=(2\pi\compi)^{-n}e^{-\langle \phi, A\bar{\phi}\rangle} d\bar{\phi}d\phi$ is integrable (see, e.g., \cref{eq:symmetrize}).

\vspace{.3em} \begin{lem}
\label{lem:nonneg}
Let $G$ be an inverse M-matrix with a positive definite symmetric part. For any nonnegative function $f:\mathbb{R}_+^n\rightarrow \mathbb{R}_+$ with subexponential growth, the integral $\int_{\mathbb{C}^n}\, \phi_a\bar{\phi}_b f(\phi\bar{\phi}) d\mu(\phi)$ is finite and nonnegative for any $a,b\in [n]$.
\end{lem}

\begin{proof}
To see that
the integral is nonnegative, note that
    \begin{align}   \label{eq:appceqn}
    \int_{\mathbb{C}^n}\, \phi_a\bar{\phi}_b f(\phi\bar{\phi}) d\mu(\phi) & = \frac{1}{(2\pi\compi)^n}\int_{\mathbb{C}^n}\, \phi_a\bar{\phi}_b f(\phi\bar{\phi}) e^{-\sum_{ij} A_{ij}\phi_i\bar{\phi}_j} d\bar{\phi}d\phi \nonumber\\
    & = \frac{1}{(2\pi)^n}\int_{\mathbb{R}_+^n}\int_{[0,2\pi)^n}\, \sqrt{l_a l_b}e^{\compi(\theta_a-\theta_b)} f(l) e^{-\sum_{ij} A_{ij}\sqrt{l_i l_j}e^{\compi(\theta_i-\theta_j)}} d\theta dl \nonumber\\
    & = \int_{\mathbb{R}_+^n}\sqrt{l_a l_b} f(l) e^{-\sum_{i} A_{ii}l_i}\bigg(\frac{1}{(2\pi)^n}\int_{[0,2\pi)^n}\, e^{\compi(\theta_a-\theta_b)}   e^{-\sum_{i\neq j} A_{ij}\sqrt{l_i l_j}e^{\compi(\theta_i-\theta_j)}} d\theta\bigg) dl,
    \end{align}
where we have used equation \eqref{eq:diffform} in writing the second equality.

Moreover, we have
    \begin{align*}
    e^{-\sum_{i\neq j} A_{ij}\sqrt{l_i l_j}e^{\compi(\theta_i-\theta_j)}} & = \prod_{i\neq j}e^{- A_{ij}\sqrt{l_i l_j}e^{\compi(\theta_i-\theta_j)}}\\
    & = \prod_{i\neq j}\sum_{k_{ij}\in\mathbb{N}} \frac{(- A_{ij}\sqrt{l_i l_j}e^{\compi(\theta_i-\theta_j)})^{k_{ij}}}{k_{ij}!}\\
    & = \sum_{\{k_{ij}\in\mathbb{N}, \,i\neq j\}}\prod_{i\neq j} \frac{(- A_{ij})^{k_{ij}}(l_i l_j)^{\frac{1}{2}k_{ij}}}{k_{ij}!}e^{\compi k_{ij}(\theta_i-\theta_j)}.
    \end{align*}

Fix some $l\in\mathbb{R}_+^n$. The inner integral in equation \eqref{eq:appceqn} becomes
    \begin{align*}
    \frac{1}{(2\pi)^n}\int_{[0,2\pi)^n}\, & e^{\compi(\theta_a-\theta_b)}   e^{-\sum_{i\neq j} A_{ij}\sqrt{l_i l_j}e^{\compi(\theta_i-\theta_j)}} d\theta \\
    &= \frac{1}{(2\pi)^n}\int_{[0,2\pi)^n}\, e^{\compi(\theta_a-\theta_b)}\bigg(\sum_{\{k_{ij}\in\mathbb{N}, \,i\neq j\}}\prod_{i\neq j} \frac{(- A_{ij})^{k_{ij}}(l_i l_j)^{\frac{1}{2}k_{ij}}}{k_{ij}!}e^{\compi k_{ij}(\theta_i-\theta_j)}\bigg) d\theta  \\
    &= \frac{1}{(2\pi)^n}\int_{[0,2\pi)^n}\, e^{\compi(\theta_a-\theta_b)}\bigg(\sum_{\{k_{ij}\in\mathbb{N}, \,i\neq j\}}\bigg(\prod_{i\neq j} \frac{(- A_{ij})^{k_{ij}}(l_i l_j)^{\frac{1}{2}k_{ij}}}{k_{ij}!}\bigg)e^{\sum_{i\neq j}\compi k_{ij}(\theta_i-\theta_j)}\bigg) d\theta  \\
    &= \sum_{\{k_{ij}\in\mathbb{N}, \,i\neq j\}}\bigg(\prod_{i\neq j} \frac{(- A_{ij})^{k_{ij}}(l_i l_j)^{\frac{1}{2}k_{ij}}}{k_{ij}!}\bigg)\cdot\bigg(\frac{1}{(2\pi)^n}\int_{[0,2\pi)^n}\, e^{\compi(\theta_a-\theta_b)}e^{\sum_{i\neq j}\compi k_{ij}(\theta_i-\theta_j)} d\theta\bigg).
    \end{align*}

Notice that $\frac{1}{(2\pi)^n}\int_{[0,2\pi)^n}\, e^{\compi(\theta_a-\theta_b)}e^{\sum_{i\neq j}\compi k_{ij}(\theta_i-\theta_j)} d\theta=\frac{1}{(2\pi)^n}\int_{[0,2\pi)^n}\, e^{\sum_{i\neq j}\compi r_i\theta_i} d\theta$ where $r_i:=\sum_{j\neq i} (k_{ij}- k_{ji})+\ind(i=a)-\ind(i=b)$. Since $\frac{1}{2\pi}\int_{[0,2\pi)}e^{\compi r_i\theta_i}d\theta_i=\ind(r_i=0)$, we know that $\frac{1}{(2\pi)^n}\int_{[0,2\pi)^n}\, e^{\compi(\theta_a-\theta_b)}e^{\sum_{i\neq j}\compi k_{ij}(\theta_i-\theta_j)} d\theta=\ind(r_i=0,\forall i\in[n])$. Therefore, we conclude that
    \begin{align*}
    \frac{1}{(2\pi)^n}\int_{[0,2\pi)^n}\, & e^{\compi(\theta_a-\theta_b)}   e^{-\sum_{i\neq j} A_{ij}\sqrt{l_i l_j}e^{\compi(\theta_i-\theta_j)}} d\theta \\
    &= \sum_{\{k_{ij}\in\mathbb{N}, \,i\neq j\}}\bigg(\prod_{i\neq j} \frac{(- A_{ij})^{k_{ij}}(l_i l_j)^{\frac{1}{2}k_{ij}}}{k_{ij}!}\bigg)\cdot \ind(r_i=0,\forall i\in[n])\\
    &= \sum_{\{k_{ij}\in\mathbb{N},\,i\neq j:\, r_i(k)=0,\forall i\in[n]\}}\bigg(\prod_{i\neq j} \frac{(- A_{ij})^{k_{ij}}(l_i l_j)^{\frac{1}{2}k_{ij}}}{k_{ij}!}\bigg).
    \end{align*}

The original integral should be
    \begin{align*}
    \int_{\mathbb{C}^n}\, & \phi_a\bar{\phi}_b f(\phi\bar{\phi}) d\mu(\phi) \\
    & = \int_{\mathbb{R}_+^n}\sqrt{l_a l_b} f(l) e^{-\sum_{i} A_{ii}l_i} \cdot \sum_{\{k_{ij}\in\mathbb{N},\,i\neq j:\, r_i(k)=0,\forall i\in[n]\}}\bigg(\prod_{i\neq j} \frac{(- A_{ij})^{k_{ij}}(l_i l_j)^{\frac{1}{2}k_{ij}}}{k_{ij}!}\bigg) dl\geq 0.
    \end{align*}

Here, we used the sign property of M-matrices, which guarantees $-A_{ij}\geq 0, \forall i\neq j$. We now justify the series expansions and the integration-summation exchanges, critically using positive definiteness. Note that we have
    \begin{align*}
    \int_{\mathbb{R}_+^n}\sqrt{l_a l_b} & f(l) e^{-\sum_{i} A_{ii}l_i}\bigg(\frac{1}{(2\pi)^n}\int_{[0,2\pi)^n}\, \prod_{i\neq j}\sum_{k_{ij}\in\mathbb{N}} \bigg|e^{\compi(\theta_a-\theta_b)}\frac{(- A_{ij}\sqrt{l_i l_j}e^{\compi(\theta_i-\theta_j)})^{k_{ij}}}{k_{ij}!}\bigg| d\theta\bigg) dl\\
    &= \int_{\mathbb{R}_+^n}\sqrt{l_a l_b} f(l) e^{-\sum_{i} A_{ii}l_i}\bigg(\frac{1}{(2\pi)^n}\int_{[0,2\pi)^n}\, \prod_{i\neq j}\sum_{k_{ij}\in\mathbb{N}} \frac{(-A_{ij})^{k_{ij}}(l_i l_j)^{\frac{1}{2}k_{ij}}}{k_{ij}!} d\theta\bigg) dl\\
    &= \int_{\mathbb{R}_+^n}\sqrt{l_a l_b} f(l) e^{-\sum_{i} A_{ii}l_i}\bigg(\frac{1}{(2\pi)^n}\int_{[0,2\pi)^n}\, \prod_{i\neq j}e^{-A_{ij}\sqrt{l_i l_j}} d\theta\bigg) dl\\
    &= \int_{\mathbb{R}_+^n}\sqrt{l_a l_b} f(l) e^{-\sum_{i} A_{ii}l_i} \prod_{i\neq j}e^{-A_{ij}\sqrt{l_i l_j}} dl\\
    &= \int_{\mathbb{R}_+^n}\sqrt{l_a l_b} f(l) e^{-\sum_{i} A_{ii}l_i} e^{-\sum_{i\neq j}A_{ij}\sqrt{l_i l_j}} dl\\
    &= \int_{\mathbb{R}_+^n}\sqrt{l_a l_b} f(l) e^{-\langle \sqrt{l}, A\sqrt{l}\rangle} dl.
    \end{align*}
Since $A$ has positive definite symmetric part, we know there exists $c>0$, such that $\langle \sqrt{l}, A\sqrt{l}\rangle \geq c\|\sqrt{l}\|^2=c\sum_il_i$. Therefore, the above integral should be less than $ \int_{\mathbb{R}_+^n}\sqrt{l_a l_b} f(l) e^{-c\sum_i l_i} dl.$ Now $f(l)$ having subexponential growth guarantees this integral to be finite. Then, the dominated convergence theorem guarantees the expansion and integration-summation exchanges.
\end{proof}

\section{Miscellaneous supporting results}  \label{app:misc}

\vspace{.3em} \begin{lem}
\label{lem:inv_psd}
Given a non-singular matrix $G\in M_n(\mathbb{R})$ with a positive definite symmetric part, its inverse $A:=G^{-1}$ also has a positive definite symmetric part.
\end{lem}
\begin{proof}
Note that
\[\frac{ A + A^T}{2}= \frac{G^{-1} + G^{-T}}{2} = G^{-T} \frac{G + G^T}{2} G^{-1}.\]

Now for any $x\in\mathbb{R}^n$ such that $x\neq 0$, we have
\[\bigg\langle x, \frac{ A + A^T}{2}x\bigg\rangle
=\bigg\langle x, G^{-T} \frac{G + G^T}{2} G^{-1}x\bigg\rangle =\bigg\langle G^{-1}x, \frac{G + G^T}{2}G^{-1}x\bigg\rangle> 0,\] where the last step used the assumption that $G$ has a positive definite symmetric part. Therefore, we conclude that $A$ also has a positive definite symmetric part.
\end{proof}

\section*{Acknowledgements}

This research was supported by the grants CCF-1901004 and CIF-2007965 from the U.S. National Science Foundation.

\end{document}